\documentclass[a4paper]{amsart}
\usepackage{a4wide}
\usepackage{amssymb}
\usepackage[hidelinks]{hyperref}
\usepackage{graphicx}
\usepackage{tikz}
\usepackage{overpic}
\usepackage{breakcites}
\usepackage{enumitem}
\usepackage{todonotes}
\usepackage[export]{adjustbox}

\newcommand{\Z}{\mathbb{Z}}
\newcommand{\R}{\mathbb{R}}
\newcommand{\C}{\mathbb{C}}

\newcommand{\re}{\operatorname{Re}}
\newcommand{\im}{\operatorname{Im}}

\theoremstyle{plain}
\newtheorem{theorem}{Theorem}[section]
\newtheorem{lemma}[theorem]{Lemma}

\newtheorem{proposition}[theorem]{Proposition}
\newtheorem{conjecture}[theorem]{Conjecture}

\theoremstyle{definition}

\newtheorem{remark}[theorem]{Remark}

\newtheorem*{convention}{Convention}

\def \cT {\mathcal{T}}
\def \cR {\mathcal{R}}

\DeclareMathOperator{\sn}{sn}

\begin{document}

\title[The $\mathrm{tG}$ and $\mathrm{rGL}$ families]{Existence of the tetragonal and rhombohedral deformation families of the gyroid}

\author{Hao Chen}
\address[Chen]{Georg-August-Universit\"at G\"ottingen, Institut f\"ur Numerische und Angewandte Mathematik}
\email{h.chen@math.uni-goettingen.de}
\thanks{H.\ Chen is supported by Individual Research Grant from Deutsche Forschungsgemeinschaft within the project ``Defects in Triply Periodic Minimal Surfaces'', Projektnummer 398759432.}

\keywords{Triply periodic minimal surfaces}
\subjclass[2010]{Primary 53A10}

\date{\today}

\begin{abstract}
	We provide an existence proof for two 1-parameter families of embedded triply
	periodic minimal surfaces of genus three, namely the tG family with
	tetragonal symmetry that contains the gyroid, and the rGL family with
	rhombohedral symmetry that contains the gyroid and the Lidinoid, both
	discovered numerically in the 1990s.  The existence was previously proved
	within a neighborhood of the gyroid and the Lidinoid, using Weierstrass data
	defined on branched rectangular tori.  Our main contribution is to extend the
	technique to branched tori that are not necessarily rectangular.
\end{abstract}

\maketitle

\section{Introduction}

A triply periodic minimal surface (TPMS) is a minimal surface $M \subset \R^3$
that is invariant under the action of a 3-dimensional lattice $\Lambda$.  The
quotient surface $M/\Lambda$ then lies in the flat 3-torus~$\R^3/\Lambda$.  The
genus of $M/\Lambda$ is at least three, and TPMSs of genus three are
abbreviated as TPMSg3s.

Due to their frequent appearance in nature and science, the study of TPMSs
enjoys regular contributions from physicists, chemists, and crystallographers.
Their discoveries of interesting examples often precede the rigorous
mathematical treatment.  The most famous example would be the gyroid discovered
in 1970 by Alan Schoen~\cite{schoen1970}, then a scientist at NASA.  Unlike
other TPMSs known at the time, the gyroid does not contain any straight line or
planar curvature line, hence it cannot be constructed by the popular conjugate
Plateau method~\cite{karcher1989}.  The second TPMSg3 with this property was
discovered only twenty years later in 1990, by chemists Lidin and
Larsson~\cite{lidin1990}, and known nowadays as the Lidinoid.  The gyroid and
the Lidinoid were later proved to be embedded by mathematicians Gro\ss
e-Brauckmann and Wohlgemuth~\cite{kgb1996}.

\medskip

By intentionally reducing symmetries of the gyroid and the Lidinoid, two
1-parameter families of TPMSg3s, which we call tG and rGL, were discovered in
\cite{fogden1993, fogden1999}; see also~\cite{schroderturk2006}.  Both families
contain the gyroid and retain respectively its rhombohedral and tetragonal
symmetries.  Remarkably, none of these surfaces contains straight lines or
planar curvature lines.  Moreover, tG and rGL surfaces are not contained in the
5-parameter family of TPMSg3s constructed by Meeks~\cite{meeks1990}.  Today,
the only other explicitly known TPMSg3s outside Meeks' family are the recently
discovered 2-parameter families oH (containing Schwarz' H)~\cite{chen2018b} and
o$\Delta$~\cite{chen2018a}.

In~\cite{fogden1993, fogden1999}, periods were closed only numerically,
producing convincing images that leave no doubt for the existence of tG and
rGL.  Although the importance of numerical discoveries could never be
overestimated, the lack of a formal existence proof (that does not involve any
numerics) often indicates room for better mathematical understanding.  Indeed,
our approach in the current paper brings new ways to visualize the tG and rGL
surfaces.

\medskip

An attempt of existence proof for tG and rGL was carried out by
Weyhaupt~\cite{weyhaupt2006, weyhaupt2008}.  He used the flat structure
technique first introduced by Weber and Wolf~\cite{weber1998, weber2002}.
Unlike~\cite{fogden1993, fogden1999} who parameterized TPMSg3s on branched
spheres, Weyhaupt defined Weierstrass data on branched tori.  In particular,
the gyroid and the Lidinoid, as well as the classical Schwarz' surfaces, are
parameterized on rectangular tori.

Weyhaupt showed that there exists a continuous 1-parameter family of tori that
solve the period problems for tG.  This family contains the rectangular torus
of the gyroid and ``does not deform too much from rectangular''.
See~\cite[Lemmas~4.3 \& 4.5]{weyhaupt2008}\footnotemark\ for precise
statements.  Similar results were obtained for an ``rG'' family near the
gyroid and an ``rL'' family near the Lidinoid.  These conclusions, in
Weyhaupt's own word~\cite[\S 6.0.6]{weyhaupt2006}, only asserted ``the
existence of an analytic family of possibly small parameter space''.  In
particular, it was not clear that rL and rG are part of the same family, which
we call rGL in the current paper.  Weyhaupt was aware that, to get away from
small neighborhoods, one needs to deal with Weierstrass data defined on
non-rectangular tori.

\footnotetext{\cite[Lemma~4.4]{weyhaupt2008} is a typo.  Weyhaupt meant $b>0$ on $B$ and $b<0$ on $Y$.}

\medskip

In the current paper, we provide an existence proof for the whole tG and rGL
families.  More precisely, our main results are
\begin{theorem}[tG]\label{thm:maintG}
	There is a 1-parameter family of TPMSg3s containing the gyroid, which we call
	tG, with the following properties:
	\begin{itemize}
		\item Each TPMSg3 in tG admits a screw symmetry of order 4 around a
			vertical axis and rotational symmetries of order 2 around horizontal
			axes.

		\item tG intersects the tD family and tends to 4-fold saddle towers in the
			limit.
	\end{itemize}
\end{theorem}

\begin{theorem}[rGL]\label{thm:mainrGL}
	There is a 1-parameter family of TPMSg3s containing the gyroid and the
	Lidinoid, which we call rGL, with the following properties:
	\begin{itemize}
		\item Each TPMSg3 in rGL admits a screw symmetry of order 3 around a
			vertical axis and rotational symmetries of order 2 around horizontal
			axes.

		\item rGL intersects the rPD family and tends to 3-fold saddle towers in
			the limit.
	\end{itemize}
\end{theorem}

Two properties are listed in each of the statements above.  The first specifies
the expected symmetries.  Weyhaupt~\cite{weyhaupt2006, weyhaupt2008} has proved
that TPMSg3s with these symmetries exist in a neighborhood of the gyroid and
the Lidinoid.  The significance of our work lies in the second property, which
states that the 1-parameter family continues in one direction until
intersecting Meeks' family and, in the other direction, to degenerate limits
where curvature blows up.

\medskip


Let us give a preview of our approach.

We will use the same Weierstrass parameterization as Weyhaupt, only more
explicit in terms of the Jacobi $\sn$ function.  Explicit computations are not
possible for non-rectangular tori.  However, we notice from the Weierstrass
data that the associate family of every tG or rGL surface contains a ``twisted
catenoid''.  These are minimal annuli bounded by curved squares or triangles.
Then we generalize a point of view from~\cite{kgb1996}: As one travels along
the associate family, the twisted catenoids open up into gyrating ribbons, and
the surface is immersed if adjacent ribbons ``fit exactly into each other''.
This leads to two expressions for the associate angle, and the period problem
asks to find tori for which the two expressions are equal.

It turns out that the torus of an rG or tGL surface is well defined up to a
hyperbolic reflection group.  The boundary of its fundamental domain
corresponds to classical Schwarz' TPMSg3s, which we understand very well.  This
already allows us to conclude the existence of the families, all the way to the
degenerate limits.  We then investigate the asymptotic behavior of the period
problem at the limits of the tD or rPD family.  This allows us to locate the
intersections with Meeks' family.  A uniqueness statement hidden in Weyhaupt's
work~\cite{weyhaupt2006, weyhaupt2008} implies that the families must contain
the gyroid, whose embeddedness then ensures the embeddedness of all TPMSg3s in
the families.

\medskip

The paper is organized as follows.

In Section~\ref{sec:symmetry}, we describe the symmetries of tG and rGL
surfaces.  This is done by relaxing a rotational symmetry of the classical tP,
H and rPD surfaces to a screw symmetry of the same order.  We will define a
family $\cT$ of TPMSg3s with order-4 screw symmetries, which contains the tG
family as well as Schwarz' classical tP, tD and CLP families; we also define a
family $\cR$ of TPMSg3s with order-3 screw symmetries, which contains the rGL
family as well as the classical H and rPD families.

In Section~\ref{sec:weierstrass}, we deduce the Weierstrass data for surfaces
in $\cT$ and $\cR$ from their symmetries.  We first prove that surfaces with
screw symmetries can be represented as branched covers of flat tori.  The
symmetries then force the branch points at 2-division points of the tori.  This
allows us to write down the Weierstrass data explicitly in terms of the Jacobi
elliptic function $\sn$.  In the end of this section, we establish a convention
on the choice of the torus parameter.

In Section~\ref{sec:period}, we reduce the period problems to just one equation
using the twisted catenoids and the ribbon picture.  We then give the existence
proof in Section~\ref{sec:proof}.  The asymptotic analysis is technical hence
postponed to the Appendix.

Section~\ref{sec:discuss} is dedicated to discussions.  We point out that tG
and rGL families provide bifurcation branches that were missing
in~\cite{koiso2014}.  We also observe reflection groups that act on $\cT$ and
$\cR$, which provide new ways to visualize the known TPMSg3s.  This motivates
us to conjecture that the known surfaces are the only surfaces in $\cT$ and
$\cR$.  The first step for proving the conjecture is to confirm a uniqueness
statement for tG and rGL surfaces.

\medskip

We assume some familiarity of the reader with classical TPMSg3s including
Schwarz' surfaces, the gyroid, and the Lidinoid.  If this is not the case, it is
recommended to take a look into~\cite{fogden1992, kgb1996, weyhaupt2006}.

\subsection*{Acknowledgement}

The author is grateful to Matthias Weber for constant and helpful
conversations.  I also thank the anonymous referee, whose useful suggestions to
a preliminary version lead to significant improvement of the paper.

\section{Symmetries}\label{sec:symmetry}

The tG and rGL families were discovered by relaxing the symmetries of classical
surfaces~\cite{fogden1993, fogden1999}.  It is then a good idea to first recall
some classical TPMSg3s, namely the 1-parameter families tP, rPD and H.  We
recommend the following way to visualize.
\begin{description}[leftmargin=*]
	\item[tP]  Consider a square catenoid, i.e.\ a minimal annulus bounded by two
		horizontal squares related by a vertical translation.  Then the order-2
		rotations around the edges of the squares generate a tP surface, and any tP
		surface can be generated in this way.

	\item[H]  Consider a triangular catenoid, i.e.\ a minimal annulus bounded by
		two horizontal equiangular triangles related by a vertical translation.
		Then the order-2 rotations around the edges of the triangles generate an H
		surface, and any H surface can be generated in this way.

	\item[rPD]  Similar to the H surfaces, the only difference being that one
		bounding triangle is reversed.
\end{description}

\begin{figure}[hb]
	\includegraphics[width=0.25\textwidth,valign=t]{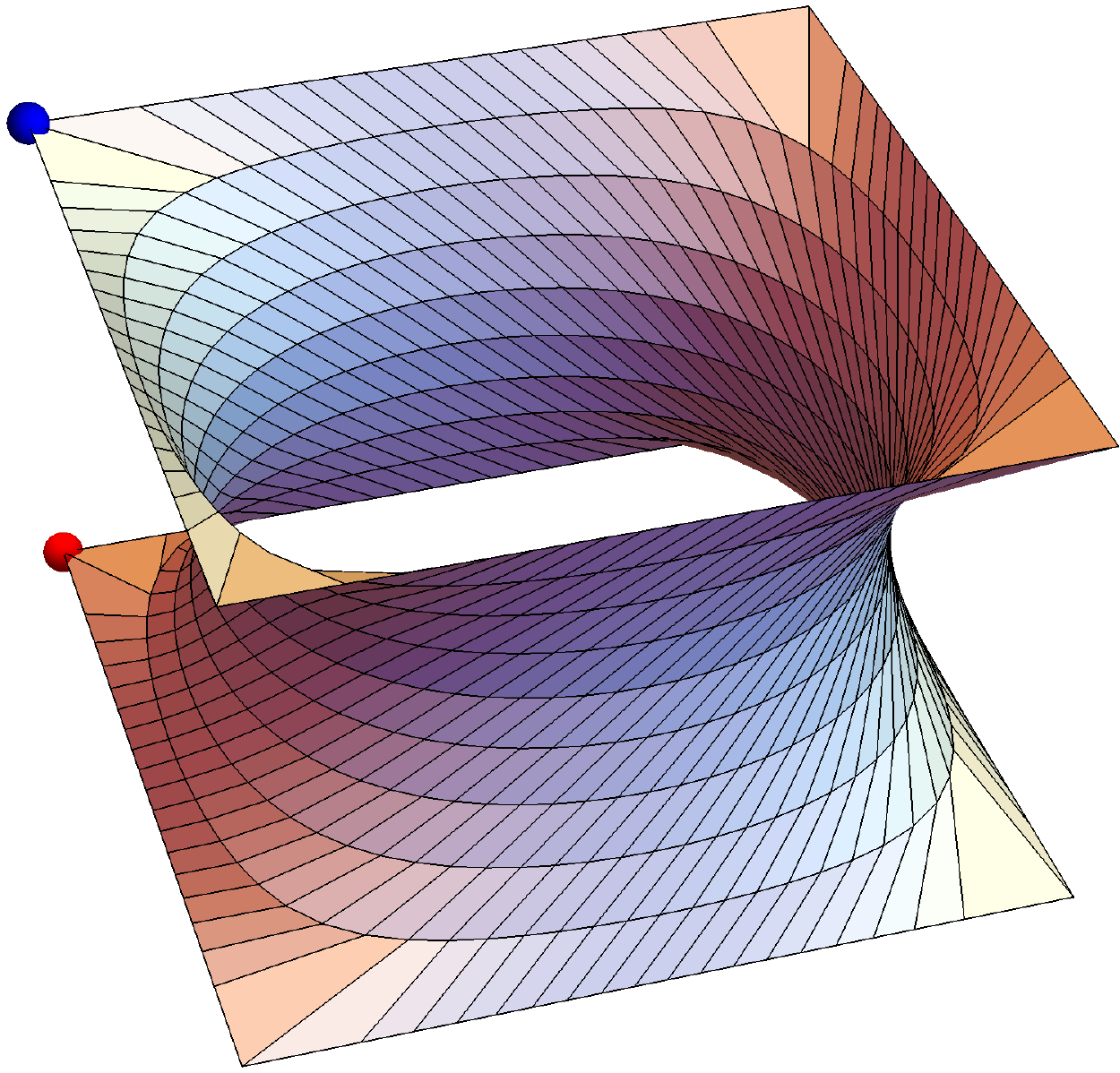}
	\includegraphics[width=0.25\textwidth,valign=t]{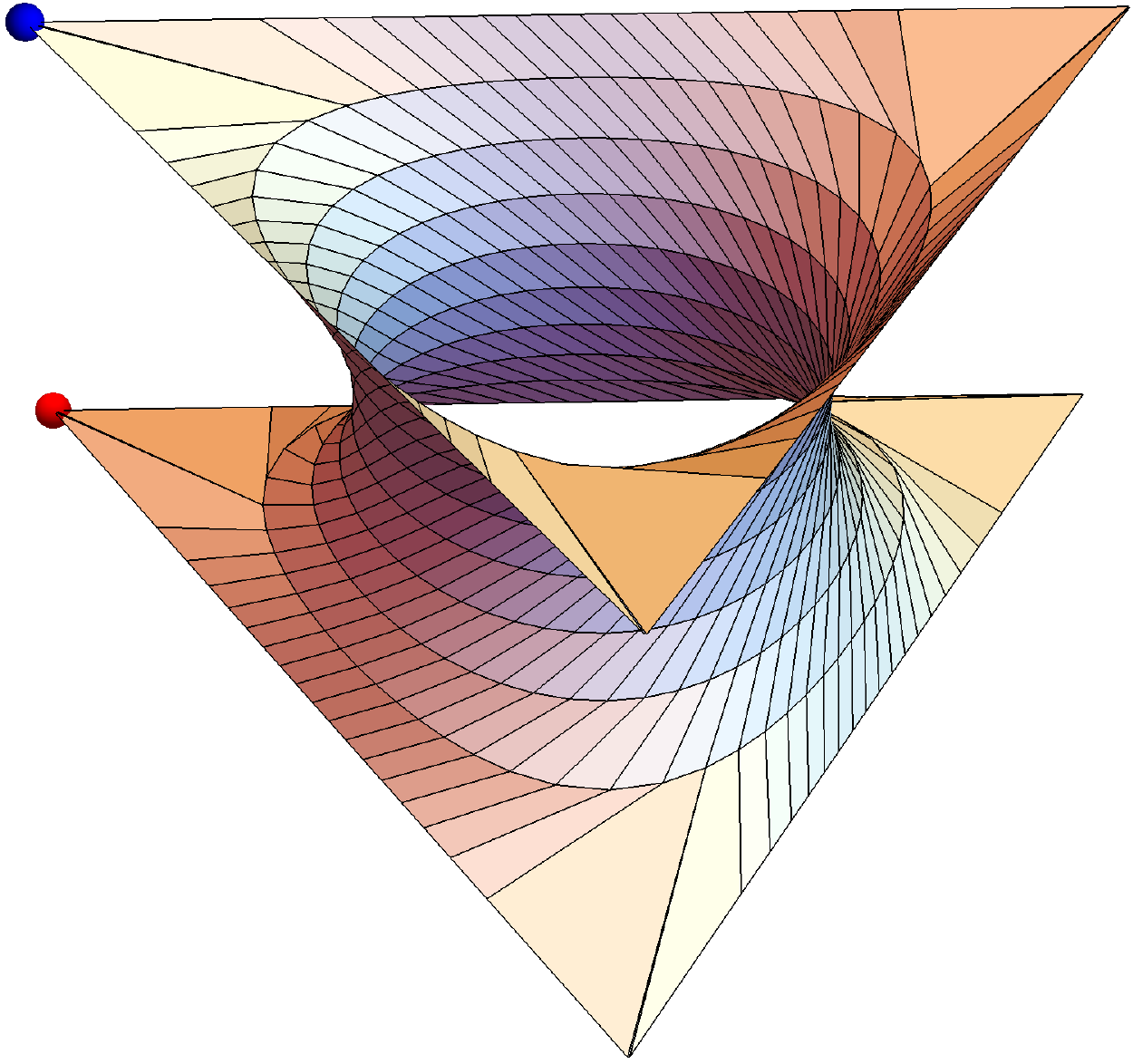}
	\includegraphics[width=0.25\textwidth,valign=t]{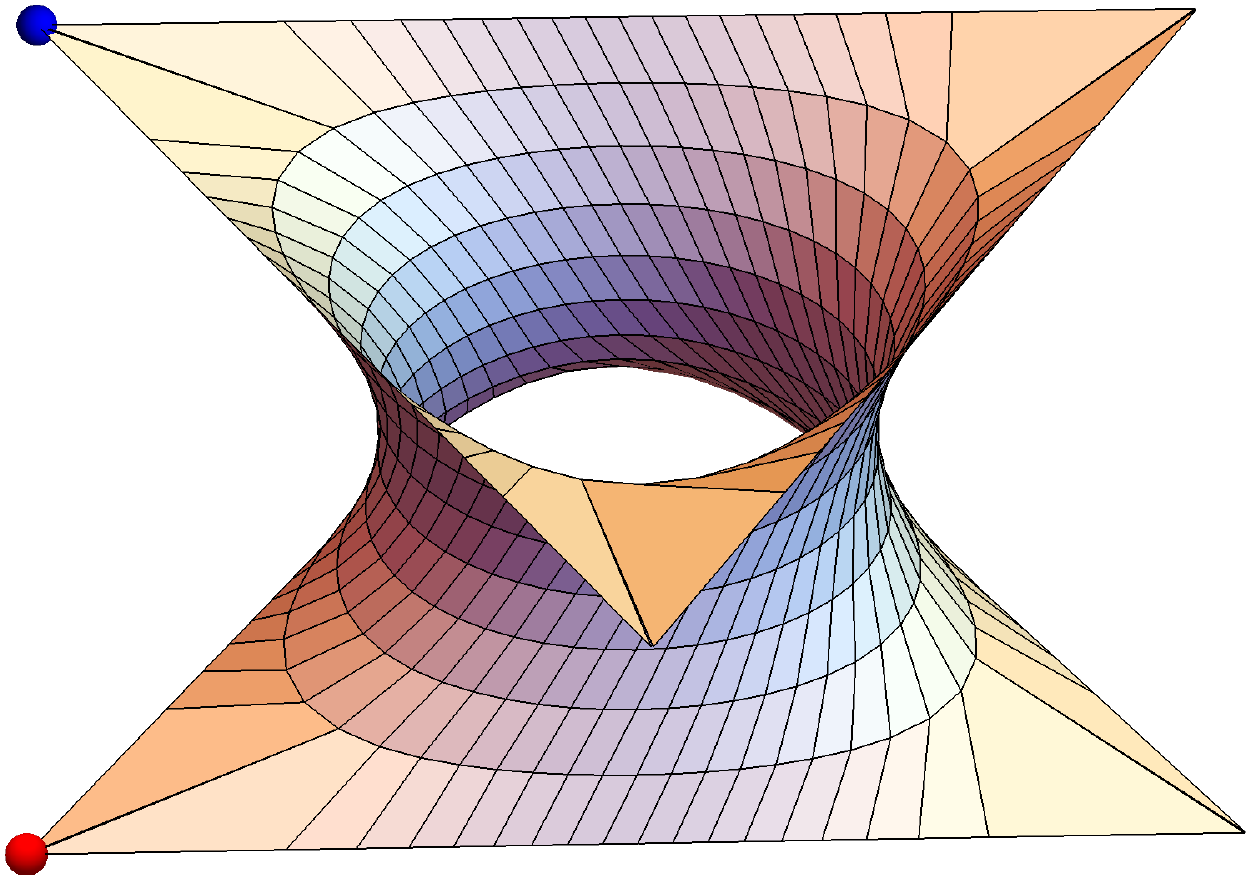}
	\caption{
		The catenoids that generate tP (left), H (middle) and rPD (right) surfaces.
	}
	\label{fig:catenoids}
\end{figure}

The catenoids that generate tP, H and rPD surfaces are shown in
Figure~\ref{fig:catenoids}.  The 1-parameter families are obtained by
vertically ``stretching'' the catenoids.  These families are, remarkably,
already known to Schwarz~\cite{schwarz1890}.  Reflections might be their most
famous and obvious symmetries.  But we highlight the following symmetries that
will help understanding tG and rGL surfaces.

\begin{description}[leftmargin=*]
	\item[Inversions] These are orientation-reversing symmetries shared by every
		TPMSg3.  Meeks~\cite{meeks1990} proved that every TPMSg3 $M/\Lambda$ has
		eight inversion centers at the points where Gaussian curvature vanishes.
		We will come back on that later.

		For tP surfaces, the eight inversion centers are at the middle points of
		the eight bounding edges.  For H and rPD surfaces, they are at the middle
		points of the six bounding edges and the two vertices (up to $\Lambda$) of
		the bounding triangles.

		The catenoids above can be seen as the quotient of $M/\Lambda$ over the
		inversions.  So the fundamental unit $M/\Lambda$ consists of two catenoids.
		See the top parts of Figures~\ref{fig:Ttorus} and~\ref{fig:Rtorus} for how
		their boundaries identify.

	\item[Order-2 rotations around horizontal axes] These are
		orientation-preserving symmetries that swap the bounding squares or
		triangles.  So their axes lie in the middle horizontal plane, parallel to
		the edges and diagonals of the bounding squares or triangles.  Up to the
		inversions, each tP surface has four such rotations, and each H or rPD
		surface has three.  

	\item[Rotations around vertical axes]  Up to the inversions, each of these
		TPMSg3s has exactly one such rotation.  The vertical axis passes through
		the centers of the bounding squares or triangles.  This symmetry is
		orientation-preserving.  Its order is 4 for tP and 3 for H and rPD.

	\item[Roto-reflections] A roto-reflection is composed of a rotation around
		the normal vector at a vertex of the bounding squares or triangles,
		followed by a reflection in the tangent plane at this vertex.  Note that
		these vertices have vertical normal vectors, hence are poles and zeros of
		the Gauss map.  These symmetries are orientation-reversing.  Their order is
		4 for tP and 6 for H and~rPD.
\end{description}

These symmetries are not independent.

\begin{itemize}
	\item The rotation around the vertical axis is the composition of two order-2
		rotations around horizontal axes.

	\item Conversely, in the presence of the rotation around the vertical axis,
		one order-2 rotational symmetry around a horizontal axis implies all the
		other.

	\item The roto-reflection arises from the screw symmetry and the inversions.

	\item For the H and rPD surfaces, \cite[Proposition~3.12]{weyhaupt2006}
		asserts that the order-3 rotation around the vertical axis implies the
		inversion symmetries in the vertices of the triangles.
\end{itemize}

Hence for TPMSg3s, the symmetries listed above can all be recovered from only
two symmetries: the rotational symmetry (of order 3 or 4) around the vertical
axis and an order-2 rotational symmetry around a horizontal axis.

\medskip

It was observed in~\cite[Lemma~4]{kgb1996} that, as one travels along the
associate family, all these symmetries are preserved, except for the rotational
symmetry around the vertical axis, which is reduced to a screw symmetry.
Recall that a screw transform is composed of a rotation and a translation in
the rotational axis.  The reduction of symmetry can be seen by noticing that
the horizontal rotation axes are no longer in the same horizontal plane, hence
their compositions induce screw transforms, instead of rotations.  Note that
the interdependences of the symmetries remain the same after this reduction.
In particular, the argument for~\cite[Proposition~3.12]{weyhaupt2006} applies
word by word to the order-3 screw symmetry.

The gyroid is in the associate family of Schwarz' P surface, which is in the
intersection of tP and rPD families.  Hence the gyroid admits a screw symmetry
around a vertical axis and an order-2 rotational symmetry around a horizontal
axis.  The order of the screw symmetry is 3 or~4, depending on which rotational
axes of P is placed vertically.  Similarly, the Lidinoid is in the associate
family of an H surface, hence admits an order-3 screw symmetry around a
vertical axis and an order-2 rotational symmetry around a horizontal axis.  As
we have discussed, all other symmetries listed above can be recovered from
these two.

\begin{remark}
	Interestingly, no other embedded surface is contained in the associate
	families of the tP, H or rPD surfaces.  This follows from a uniqueness
	statement hidden in the argument of Weyhaupt~\cite{weyhaupt2006}, as we will
	explain in Section~\ref{sec:discuss}.
\end{remark}

In the remaining of the paper, for the sake of a uniform treatment of Schwarz'
surfaces and the deformations of the gyroid, we will see rotations as screw
transforms with 0 translation.  The paper aims at the following two sets of
TPMSg3s
\begin{itemize}
	\item $\cT$ consists of embedded TPMSg3s that admit order-4 screw symmetries
		around vertical axes and order-2 rotational symmetries around horizontal axes.

	\item $\cR$ consists of embedded TPMSg3s that admit order-3 screw symmetries
		around vertical axes and order-2 rotational symmetries around horizontal axes.
\end{itemize}

Schwarz' tP surfaces and the gyroid belong to $\cT$.  The conjugates of the tP
surfaces, namely Schwarz' tD surfaces, belong to $\cT$.  We will see that
Schwarz' CLP surfaces also belong to $\cT$.  Our main Theorem~\ref{thm:maintG}
states that there exists another 1-parameter family in $\cT$, denoted by tG,
that contains the gyroid.

Schwarz' rPD, H surfaces, as well as the gyroid and the Lidinoid belong to
$\cR$.  Our main Theorem~\ref{thm:mainrGL} states that there exists another
1-parameter family in $\cR$, denoted by rGL, that contains the gyroid and the
Lidinoid.

\section{Weierstrass parameterization}\label{sec:weierstrass}

We use~\cite{meeks1990} for general reference about TPMSg3.

Let $M$ be a TPMS invariant under the lattice $\Lambda$.
Meeks~\cite{meeks1990} proved that $M/\Lambda$ is of genus three if and only if
it is \emph{hyperelliptic}, meaning that it can be represented as a two-sheeted
branched cover $M/\Lambda \to \mathbb{S}^2$ over the sphere.  The Gauss map $G$
provides such a branched covering.  If $M/\Lambda$ is of genus three, the
Riemann-Hurwitz formula implies eight branch points of $G$.  We call the
corresponding ramification points on $M/\Lambda$ \emph{hyperelliptic points}.
An inversion (in the ambient space $\R^3$) in any of the hyperelliptic points
induces an isometry that exchanges the two sheets.  

Let $z_1, \dots, z_8 \in \C$ be the stereographic projections of the Gauss map
at the branch points.  Now consider the hyperelliptic Riemann surface of genus
three defined by
\[
	w^2 = P(z) = \Pi_{i=1}^8(z-z_i).
\]
Then we have the following Weierstrass parameterization for $M/\Lambda$:
\begin{equation}\label{eq:weierstrass1}
	(z,w) \mapsto \re\int^{(z,w)} \frac{(1-z^2, (1+z^2)i, 2z)}{w}dz.
\end{equation}
This parameterization has been widely used for constructing TPMSg3s, ranging
from the classical example of Schwarz'~\cite{schwarz1890} to the tG and rGL
families discovered in~\cite{fogden1993, fogden1999}.

We use the following form of Weierstrass parameterization that traces back to
Osserman~\cite{osserman1964},
\begin{equation}\label{eq:weierstrass2}
	\Sigma \ni p \mapsto \re\int^p (\omega_1, \omega_2, \omega_3) = \re\int^p
	\Big(\frac{1}{2}\big(\frac{1}{G}-G\big), \frac{i}{2}\big(\frac{1}{G}+G\big),
	1\Big) dh \in \R^3.
\end{equation}
Here $\Sigma$ is a Riemann surface, on which $\omega_1, \omega_2, \omega_3$
must all be holomorphic.  In particular, the holomorphic differential $\omega_3
= dh$ is called the \emph{height differential}.  $G$ denotes (the stereographic
projection of) the Gauss map.  By comparing \eqref{eq:weierstrass1} and
\eqref{eq:weierstrass2}, one sees the correspondence $G = z$ and $dh = z dz/w$
(up to a scaling factor $2$).  The triple $(\Sigma, G, dh)$ is called
Weierstrass data.

The purpose of this section is to determine the Weierstrass data for surfaces
in $\cT$ and $\cR$ from their symmetries.  In particular, $\Sigma$ will be a
branched torus, whose branch points are determined in Lemmas~\ref{lem:branchT}
and~\ref{lem:branchR}.  The height differential is determined in
Lemma~\ref{lem:dh}, and the Gauss map is explicitly given in Lemma~\ref{lem:G}.

\subsection{Weierstrass data on tori}

Surfaces in $\cT$ and $\cR$ all admit screw symmetries.  The following
proposition justifies our choice of branched tori for $\Sigma$.
\begin{proposition}
	If a TPMSg3 $M$ admits a screw symmetry $S$, then $(M/\Lambda)/S$ is of genus
	one.
\end{proposition}
Recall that we consider rotational symmetries as special screw symmetries, for
which the proposition was proved in~\cite[Proposition~2.8]{weyhaupt2008} but
with minor flaws.  Hence we include a proof here for completeness.
\begin{proof}
	The height differential $dh$ is invariant under $S$, hence descends
	holomorphically to the quotient $(M/\Lambda)/S$.  Since there is no
	holomorphic differential on the sphere, the genus of $(M/\Lambda)/S$ cannot
	be $0$.

	Recall the Riemann--Hurwitz formula
	\[
		g = n(g' - 1) + 1 + B/2.
	\]
	In our case, $g=3$ is the genus of $M/\Lambda$, $g'$ is the genus of
	$(M/\Lambda)/S$, $n$ is the degree of the quotient map, and $B$ is the total
	branching number.  Since the order of a screw symmetry is at least two, we
	conclude immediately that $g' < 3$.

	It remains to eliminate the case $g'=2$.  Weyhaupt's argument
	in~\cite{weyhaupt2008} did not accomplish this task.  We proceed as
	follows\footnotemark.  Without loss of generality, we may assume the screw
	axis to be vertical.  As Weyhaupt argued, when $g'=2$ we have necessarily
	$n=2$, in which case $G \circ S = -G$.  In the
	parameterization~\eqref{eq:weierstrass1}, since $dh= zdz/w$ is invariant
	under $z\mapsto-z$, so must the polynomial $w^2=P(z)$, hence we can write
	$P(z) = Q(z^2)$, where $Q$ is a polynomial of degree 4.  Then $w^2 = Q(z)$
	defines the quotient surface, whose genus is $g'=1$.  So $g'=1$ is the only
	possibility.
\end{proof}

\footnotetext{This argument is communicated by Matthias Weber.}

The height differential $dh$, being a holomorphic 1-form on the torus, must be
of the form $r e^{-i\theta} dz$.  Varying the modulus $r$ only results in a
scaling.  Varying the argument $\theta$ gives the associate family, so we call
$\theta$ the \emph{associate angle}.  The following lemma then applies to all
TPMSg3s in $\cT$ and $\cR$.
\begin{lemma}\label{lem:dh}
	If a TPMSg3 $M/\Lambda$ with a screw symmetry $S$ is represented on the
	branched cover of the torus $(M/\Lambda)/S$, then up to the scaling, the
	height differential $dh$ must be the lift of $e^{-i\theta}dz$ (note the
	sign!).
\end{lemma}

\subsection{Locating branch points}

By~\cite[Lemma~2(ii)]{kuribayashi1979}, we know that the order of the screw
symmetry $S$ must be $2$, $3$ or $4$.  This follows easily from a result of
Hurwitz~\cite{hurwitz1932}, cited in~\cite{kuribayashi1979} as Lemma~1.
Moreover, if the order of $S$ is prime, the following formula
from~\cite{farkas1992} allows us to calculate the number of fixed points: 
\[
	|\operatorname{fix}(S)| = 2 + \frac{2g - 2g'\operatorname{order}(S)}{\operatorname{order}(S) - 1}.
\]
In particular, a screw symmetry of order 2 fixes exactly four points, and a
screw symmetry of order~3 fixes exactly two points.  The following lemma
follows from the same argument as in~\cite[Lemmas~3.9, 3.13]{weyhaupt2006}.

\begin{lemma}\leavevmode
	\begin{itemize}
		\item If a TPMSg3 $M/\Lambda$ admits a screw symmetry $S$ of order 2, then
			$G^2$ descends to an elliptic function on the torus $(M/\Lambda)/S$ with
			two simple zeros and two simple poles.

		\item If $M/\Lambda$ admits a screw symmetry $S$ of order 3, then $G^3$
			descends to an elliptic function on the torus $(M/\Lambda)/S$ with a
			double-order zero and a double-order pole.
	\end{itemize}
\end{lemma}

We now try to locate the branch points of the covering map for surfaces in
$\cT$ and $\cR$.  Since the ramification points on $M/\Lambda$ are all poles
and zeros of the Gauss map, our main tool is naturally Abel's Theorem, which
states that the difference between the sum of poles and the sum of zeros
(counting multiplicity) is a lattice point.

A surface in $\cT$ admits a screw symmetry $S$ of order 4, $S^2$ is then a
screw symmetry of order 2.  Recall that $S$ and the inversions induce
roto-reflectional symmetries of order 4 centered at the poles and zeros of the
Gauss map.  They descend to the quotient torus $(M/\Lambda)/S^2$ as inversions
in the branch points of the covering map.

\begin{lemma}[Compare~{\cite[Lemma 3.10]{weyhaupt2006}}]\label{lem:branchT}
	Let $M$ be a TPMSg3 admitting a screw symmetry $S$ of order 4, hence
	parameterized on a branched double cover of the torus $(M/\Lambda)/S^2$.  If
	one branch point is placed at $0$, then the other branch points must be
	placed at the three 2-division points of the torus.
\end{lemma}

\begin{proof}
	We may assume that the branch point at $0$ corresponds to a zero of $G^2$.
	If a pole $p$ is not at any 2-division points, $-p$ must be a different pole
	by the roto-reflection.  Then Abel's Theorem forces the other zero to be at a
	lattice point, which is absurd.  So both poles must be placed at 2-division
	points.  Then Abel's Theorem forces the other zero at the remaining
	2-division point.
\end{proof}

Note that the screw symmetry of order 4 descends to the quotient torus as the
translation that swaps the zeros.

\medskip

\begin{lemma}[{\cite[Lemma 3.13]{weyhaupt2006}}]\label{lem:branchR}
	Let $M$ be a TPMSg3 admitting a screw symmetry $S$ of order 3, hence
	parameterized on a branched triple cover of the torus $(M/\Lambda)/S$.  If
	one branch point is placed at $0$, then the other branch points must be
	placed at the three 2-division points of the torus.
\end{lemma}

\begin{proof}
	We may assume that the branch point at $0$ corresponds to a double-order zero
	of $G^3$.  Then the Abel's Theorem forces the double-order pole to be at a
	2-division point.
\end{proof}

\subsection{An explicit expression for the Gauss map}

The locations of poles and zeros determine an elliptic function up to a complex
constant factor.  Elliptic functions with poles and zeros at lattice points and
2-division points are famously given by Jacobi elliptic functions.  In
particular, $\sn(z;\tau)$ is an elliptic function with periods $4K$ and $2iK'$.
Its zeros lie at $0$ and $2K$, and poles at $iK'$ and $2K+iK'$.  Here, $K$ is
the complete elliptic integral of the first kind with modulus $m =
\lambda(2\tau)$, $\lambda$ is the modular lambda function, and $K' := -2 i \tau
K$.

We use~\cite{lawden1989} as the major reference for elliptic functions.  Other
useful references include~\cite{bowman1961, byrd1971, borwein1998} and NIST's
Digital Library of Mathematical Functions~\cite{DLMF}.

\begin{remark}
	Note that we do not define $K'(m) = K(1-m)$, and our $\tau$ is half of the
	traditional definition but coincides with the definition on page 226
	of~\cite{lawden1989}.
\end{remark}

\medskip

Since $\sn$ has the expected zeros and poles, we may write the Gauss map $G$
for a $\cT$ surface as
\[
	G^2 =\rho \sn(4 K z;\tau).
\]
In particular, the factor $4K$ on the variable $z$ brings the defining torus to
$\C/\langle 1, \tau \rangle$, which is more convenient for us.  The zeroes of
$G^2$ are at $0$ and $1/2$, and the poles at $\tau/2$ and $(\tau+1)/2$.  The
complex factor $\rho$ is known as the L\'opez--Ros factor~\cite{lopez1991} in
the minimal surface theory.  Varying its argument only results in a rotation of
the surface in the space, hence only the norm $|\rho|$ concerns us.

The multi-valued function on $\C/\langle 1, \tau\rangle$ is given by $G = [\rho
\sn(4 K z;\tau)]^{1/2}$.  We take the branch cuts of the square root along the
segments $[0,\tau/2]$ and $[(\tau+1)/2,\tau+1/2]$, compatible
with~\cite{weyhaupt2006}; see Figure~\ref{fig:Ttorus}.  None of the branch
points is hyperelliptic point.  Instead, the symmetry $\sn(2K-z) = \sn(z)$
reveals four other inversion centers at $1/4$, $3/4$, $1/4+\tau/2$ and
$3/4+\tau/2$, and they lift to eight hyperelliptic points.

\begin{figure}[htb]
	\includegraphics[width=0.4\textwidth]{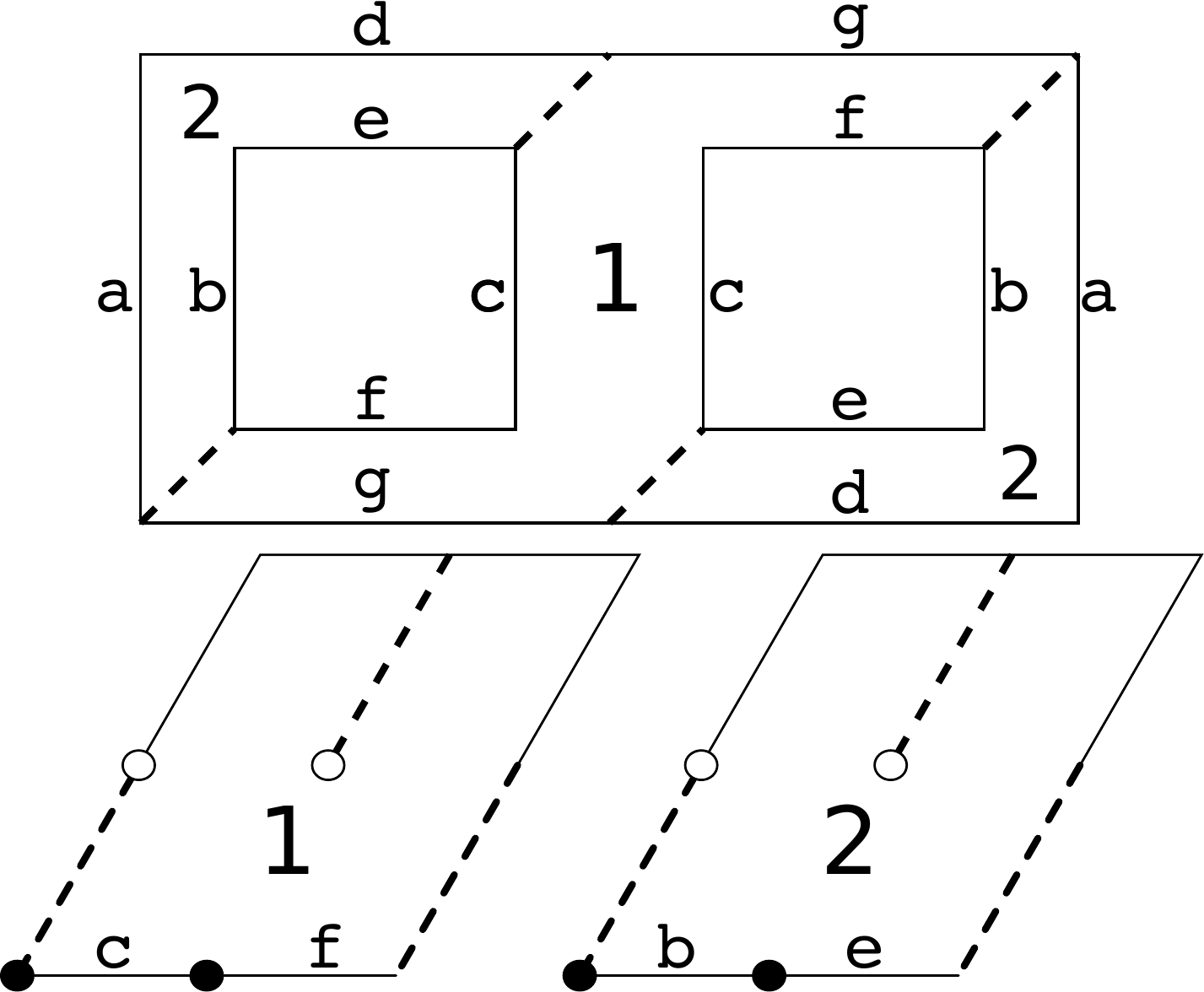}
	\caption{
		Top: Identifying segments with the same labels yields a surface of genus
		three.  The numbered regions are fundamental domains of the screw symmetry
		of order 3.  Bottom: Branch cuts in the branched torus for $\cT$ surfaces.
		Solid circles are zeros;  empty circles are poles.
	}
	\label{fig:Ttorus}
\end{figure}

Similarly, we may write the Gauss map $G$ for an $\cR$ surface as
\[
	G^3 = [\rho \sn(4 K z;\tau)]^2
\]
where $\rho$, again, is the L\'opez--Ros factor.  This function is well defined
on the torus $\C/\langle 1/2, \tau\rangle$, halving the defining torus of
$\sn(z;\tau)$.  $G^3$ has a double-order zero at $0$ and a double-order pole at
$\tau/2$, as expected.  The multi-valued function on $\C/\langle 1/2,
\tau\rangle$ is given by $G = [\rho \sn(4 K z;\tau)]^{2/3}$.  We take the
branch cuts of the cubic root along the segments $[0,\tau/2]$, $[\tau/2,\tau]$,
$[1/2,(\tau+1)/2]$ and $[(\tau+1)/2, \tau+1/2]$, compatible
with~\cite{weyhaupt2006}; see Figure~\ref{fig:Rtorus}.  This time, both branch
points are hyperelliptic points; cf.~\cite[Proposition~3.12]{weyhaupt2006}.  We
recognize two other inversion centers at $1/4$ and $1/4+\tau/2$, which lift to
six more hyperelliptic points.

\begin{figure}
	\includegraphics[width=0.4\textwidth]{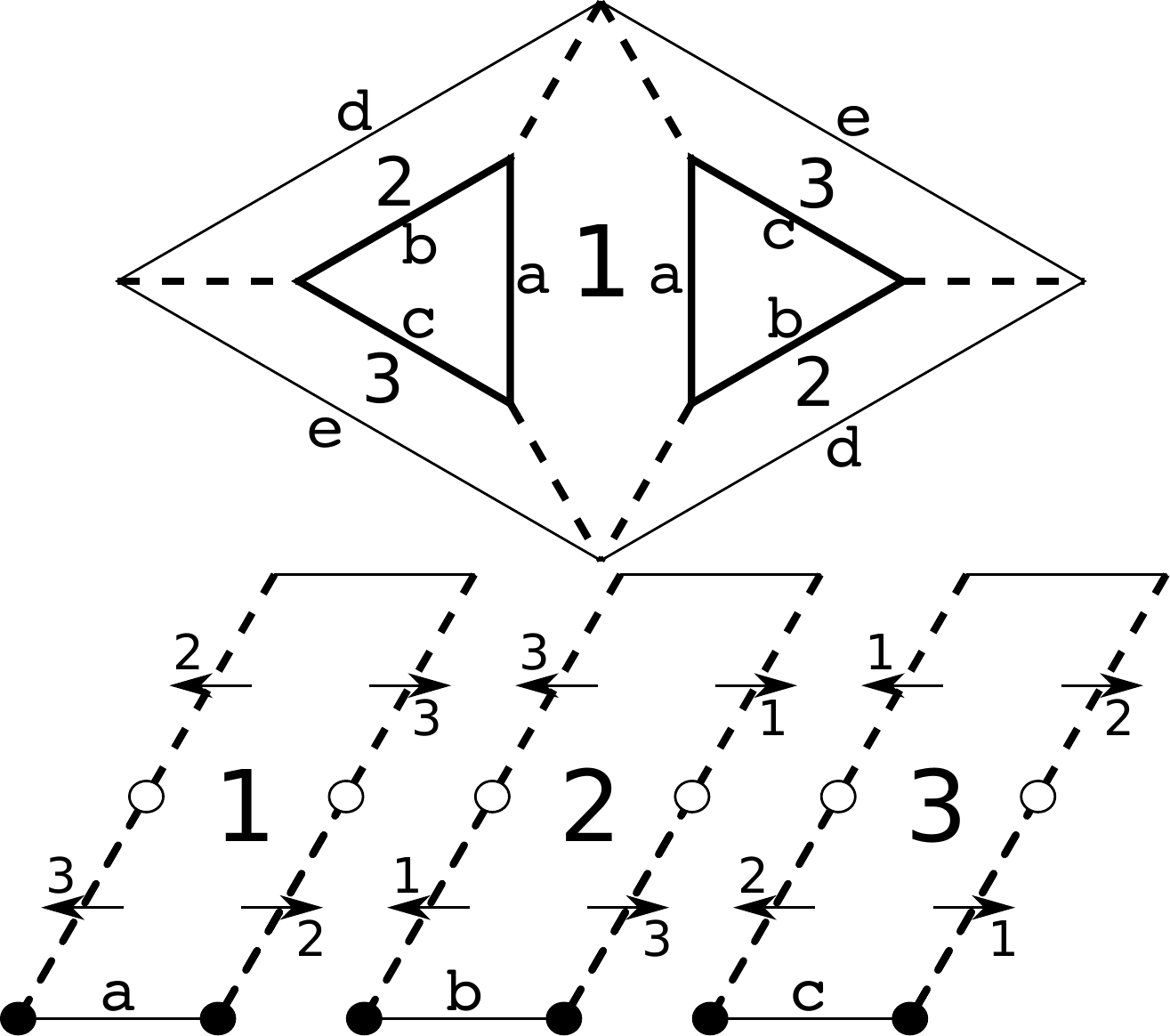}
	\caption{
		Top: Identifying segments with the same labels yields a surface of genus
		three.  The numbered regions are fundamental domains of the screw symmetry
		of order 2.  Bottom: Branch cuts in the branched torus for $\cR$ surfaces.
		Solid circles are zeros;  empty circles are poles.
	}
	\label{fig:Rtorus}
\end{figure}

\medskip

For surfaces in $\cT$ and $\cR$, the L\'opez--Ros factor $\rho$ can be
determined by the order-2 rotational symmetries that swap the poles and zeros
of the Gauss map.  More specifically, $G^2$ and $1/G^2$ for $\cT$ must have the
same residue at their respective poles, and $G^3$ and $1/G^3$ for $\cR$ must
have the same principal parts at their poles.  By the identity $\sn(z+iK';\tau)
= 1 / \sqrt{m} \sn(z;\tau)$, we deduce that $\rho = m^{1/4}$ for both $\cT$ and
$\cR$.

We have shown that
\begin{lemma}\label{lem:G}
	Up to a rotation in $\R^3$, the Gauss map $G$ have the form
	\begin{align*}
		G &= [m^{1/4} \sn(4Kz;\tau)]^{1/2} && \text{for a surface in $\cT$, and}\\
		G &= [m^{1/4} \sn(4Kz;\tau)]^{2/3} && \text{for a surface in $\cR$.}
	\end{align*}
\end{lemma}

\begin{remark}
	Equivalently, we can write the Gauss map for $\cT$ in the form
	\[
		G^2 = \rho' \frac{\theta(z)\theta(z-1/2)}{\theta(z-\tau/2)\theta(z-1/2+\tau/2)}
	\]
	where $\theta$ is the Jacobi Theta function (actually one of them) for the
	lattice $\C/\langle 1,\tau\rangle$, and the L\'opez--Ros factor $\rho' =
	e^{-i\pi(\tau-1)/2}$.  The Gauss map for $\cR$ may have a similar expression.
	This should help the readers to compare our computation with those
	in~\cite{weyhaupt2006, weyhaupt2008}.
\end{remark}

\begin{remark}
	With a change of basis for the torus, one could, of course, use other Jacobi
	elliptic functions to express the Gauss map.  We notice that the normalized
	Jacobi function $m^{1/4}\sn(z;\tau)$ is one of the three Jacobi-type elliptic
	functions constructed in~\cite[\S~3]{karcher1993}, where the symmetry is
	thoroughly investigated.
\end{remark}

\subsection{Conventions on the torus}\label{sec:convention}

Recall that the modular lambda function is invariant under the congruence
subgroup $\Gamma(2)$.  Hence $m=\lambda(2\tau)$ and $\sn(z;\tau)$ are invariant
under the congruence subgroup $\Gamma_0(4)$ generated by the transforms $\tau
\mapsto \tau+1$ and $\tau \mapsto \tau/(1-4\tau)$.  So the Weierstrass data is
invariant under $\Gamma_0(4)$, and we have infinitely many choices of $\tau$
for each surface in $\cT$ or $\cR$.  In fact, using Jacobi $\sn$ function
already limits the choice.

Surfaces in~\cite{weyhaupt2006, weyhaupt2008} all admit reflectional
symmetries, making it possible to choose $\tau$ to be pure imaginary, so the
torus is rectangular (and a different elliptic function should be used).  We
could not use this convention since, in general, the tG and rGL surfaces do not
admit reflectional symmetry, nor do surfaces in their associated families.

\medskip

The fundamental domain of $\Gamma(2)$ is usually taken as the region $C$
bounded by the vertical lines $\re \tau = \pm 1$ and the half circles $|\tau
\pm 1/2| = 1/2$.  It is then natural to choose $2\tau$ in this region, so
$\tau$ is in the region $C/2$ bounded by $\re \tau = \pm 1/2$ and $|\tau \pm
1/4| = 1/4$, a fundamental domain of $\Gamma_0(4)$.  Under this convention, our
$K'$ coincides with the standard \emph{associated} elliptic integral of the
first kind with modulus $m$; then we can safely employ the formula in most
textbooks.

This natural choice is, however, not convenient for analyzing the tG and rGL
families.  For instance, the tP and tD families correspond to the same vertical
line $\re\tau = 0$, and the gyroid also lies on this line.  In fact, we will
see that, for each $r \in (-1/2,0) \cup (0,1/2]$, there are two members of tG
with $\re\tau = r$ under the natural convention.  The same happens for rGL.
Hence we need a different convention.

\medskip

We will see that, for a surface in $\cT$ or $\cR$, the poles and zeros of the
Gauss map are aligned along vertical lines, alternatingly arranged and equally
spaced.  For the sake of a uniform treatment, we make the following convention:

\begin{convention}
	For surfaces in $\cT$ and $\cR$, we assume that the Weierstrass
	parameterization maps $(1+\tau)/2$ directly above $0$.
\end{convention}

Our convention should be seen as a marking on the surface.  Although the
Weierstrass data is invariant under the transform $\tau \mapsto \tau+1$, the
marked branch point $(1+\tau)/2$ is however different.  The marking then serves
to distinguish, for instance, Schwarz' P, D surfaces and the gyroid.  We will
see that, under our convention, $\re \tau = 0$ corresponds to the gyroid in
$\cT$ and the Lidinoid in~$\cR$.

We list in Table~\ref{tbl:weierstrass} the Weierstrass data of classical TPMSg3
under our convention.  For each family, we specify the possible $\tau$ for the
torus and the associate angle $\theta$.  We also accompany a diagram, showing
the possible $\tau$ (dashed curve), the fundamental parallelogram for a typical
example, the poles (empty circles) and the zeros (solid circles) of
$\sn(z;\tau)$ (whose defining torus could twice the shown quotient torus!), and
an arrow indicating $dh$ by pointing to the direction of increasing height.
The bottom-left corner of the parallelogram is always 0, the bottom edge
represents 1 for $\cT$ or 1/2 for $\cR$, and the left edge always represents
$\tau$.  The tori used in~\cite{weyhaupt2006, weyhaupt2008} are shown as dotted
rectangles for reference.

\begin{table}[hbt]
	\begin{tabular}{r l l l}
		tP: & $\re \tau = -1$, & $\theta = \pi/2$, &
		\tikz[baseline=(current bounding box.center),scale = 2]{
			\fill[fill=lightgray] (0,0)--(-1,0.8)--(0,0.8)--(1,0)--cycle;
			\draw[dotted] (0,0)--(0,0.8)--(1,0.8)--(1,0)--cycle;
			\filldraw (0,0) circle [radius = .05];
			\filldraw (0.5,0) circle [radius = .05];
			\filldraw[fill=white] (-0.5,0.4) circle [radius = .05];
			\filldraw[fill=white] (0,0.4) circle [radius = .05];
			\draw[very thick,->] (0, 0)--(0,0.3);
			\draw[dashed] (-1,0)--(-1,1);
		}\\[1cm]

		tD: & $\re \tau = 1$, & $\theta = 0$, &
		\tikz[baseline=(current bounding box.center),scale = 2]{
			\fill[fill=lightgray] (0,0)--(1,0.8)--(2,0.8)--(1,0)--cycle;
			\draw[dotted] (0,0)--(0,0.8)--(1,0.8)--(1,0)--cycle;
			\filldraw (0,0) circle [radius = .05];
			\filldraw (0.5,0) circle [radius = .05];
			\filldraw[fill=white] (0.5,0.4) circle [radius = .05];
			\filldraw[fill=white] (1,0.4) circle [radius = .05];
			\draw[very thick,->] (0,0)--(0.3,0);
			\draw[dashed] (1,0)--(1,1);
		}\\[1cm]

		CLP: & $|\tau + 1/2| = 1/2$, & $\theta = \arg\tau - \pi/2$, &
		\tikz[baseline=(current bounding box.center),scale = 2]{
			\fill[fill=lightgray] (0,0)--(-0.36,0.48)--(0.64,0.48)--(1,0)--cycle;
			\draw[dotted] (0,0)--(-0.36,0.48)--(0.28,0.96)--(0.64,0.48)--cycle;
			\filldraw (0,0) circle [radius = .05];
			\filldraw (0.5,0) circle [radius = .05];
			\filldraw[fill=white] (-0.18,0.24) circle [radius = .05];
			\filldraw[fill=white] (0.32,0.24) circle [radius = .05];
			\draw[very thick,->] (0,0)--(0.24,0.18);
			\draw[dashed] (0,0) arc (0:180:0.5);
		}\\[1cm]

		H: & $\re \tau = -1$, & $\theta = \pi/2$, &
		\tikz[baseline=(current bounding box.center),scale = 2]{
			\fill[fill=lightgray] (0,0)--(-1,0.8)--(-0.5,0.8)--(0.5,0)--cycle;
			\draw[dotted] (0,0)--(0,0.8)--(0.5,0.8)--(0.5,0)--cycle;
			\filldraw (0,0) circle [radius = .05];
			\filldraw (0.5,0) circle [radius = .05];
			\filldraw[fill=white] (-0.5,0.4) circle [radius = .05];
			\filldraw[fill=white] (0,0.4) circle [radius = .05];
			\draw[very thick,->] (0,0)--(0,0.3);
			\draw[dashed] (-1,0)--(-1,1);
		}\\[1cm]

		rPD: & $\re \tau = 1/2$, & $\theta = \pi/2$ or $0$, &
		\tikz[baseline=(current bounding box.center),scale = 2]{
			\fill[fill=lightgray] (0,0)--(0.5,0.8)--(1.0,0.8)--(0.5,0)--cycle;
			\draw[dotted] (0,0)--(0,0.8)--(0.5,0.8)--(0.5,0)--cycle;
			\filldraw (0,0) circle [radius = .05];
			\filldraw (0.5,0) circle [radius = .05];
			\filldraw[fill=white] (0.25,0.4) circle [radius = .05];
			\filldraw[fill=white] (0.75,0.4) circle [radius = .05];
			\draw[very thick,->] (0,0)--(0.3,0);
			\draw[dashed] (0.5,0)--(0.5,1);
		}\\[1cm]
	\end{tabular}
	\caption{Weierstrass data for classical TPMSg3.\label{tbl:weierstrass}}
\end{table}

\section{Period conditions}\label{sec:period}

Rectangular tori are convenient in many ways.  For example, many straight
segments in the branched torus correspond to geodesics on the surface, making
it possible to compute explicitly~\cite{weyhaupt2006}.  With general tori, we
lose all the nice properties.  An explicit computation is indeed hopeless, but
we are still able to say something.

\subsection{Twisted catenoids}

Assume $dh = e^{-i\pi/2} = -idz$ for the moment.  We now study the image under
the maps
\begin{align*}
	\Phi_i: \C &\to \C \qquad i=1,2\\
\Phi_1: z &\mapsto \int^z dh\cdot G\\
\Phi_2: z &\mapsto \int^z dh / G.
\end{align*}

Let us first look at the lower half of the branched torus of $\cT$, i.e.\ the
part with $0 < \im z < \im \tau/2$.  This is topologically an annulus and lifts
to its universal cover $\{ z | 0 < \im z < \im \tau/2 \}$, which is a strip in
$\C$.  By analytic continuation, $G$ lifts to a function of period $2$ on the
strip.  The boundary lines $\im z = 0^+$ and $\im z = \im \tau/2-0^+$ are then
mapped by $\Phi_1$ into periodic or closed curves.

By the same argument as in the standard proof of the Schwarz--Christoffel
formula (see also \cite{weyhaupt2006, fujimori2009}), we see that the curves
make an angle at each branch point.  The interior angle is $\pi/2$ at the poles
of $G^2$ and $3\pi/2$ at the zeros.  By the symmetry
$\sn(2K+z;\tau)=-\sn(z;\tau)$, the image of the segments $[n/2,(n+1)/2]$, $n
\in \Z$, are all congruent, and images of adjacent segments only differ by a
rotation of $\pi/2$ around their common vertex.  Moreover, by the symmetry
$\sn(2K-z;\tau)=\sn(z;\tau)$, the image of each segment admits an inversional
symmetry.  The inversion center is the image of the hyperelliptic points, at
the midpoint of each segment.  The same can be said about the segments
$[(n+\tau)/2, (n+1+\tau)/2]$, $n \in \Z$.

Therefore, the boundaries of the strip are mapped to closed curves with a
rotational symmetry of order $4$.  The curves look like ``curved squares'',
obtained by replacing the straight edges of the square by congruent copies of a
symmetric curve.  Moreover, the two twisted squares share the same rotation
center.  The strip is then mapped by $\Phi_1$ to the twisted annulus bounded
by the curved squares.  See Figure~\ref{fig:Tannulus}.

\medskip

A similar analysis can be carried out on the upper annulus $\im \tau/2 < \im z
< \im \tau$.  However, the branch cuts ensure that the two annuli continuous
into different branches (different signs of square root) when crossing the
segments $[n/2, n/2+\tau]$.  As a consequence, the boundary curves is turning
in the opposite direction.  As one passes through the segment
$[\tau/2, (\tau+1)/2]$, the images of $\Phi_1$ is extended by an inversion,
which reverses orientation.

Because of our choice of the L\'opez--Ros parameter, the image of the lower
annulus under $\Phi_2$ is congruent to the image of the upper annulus under
$\Phi_1$.  The only difference is that the inner and the outer boundaries of
the annulus are swapped.  See Figure~\ref{fig:Tannulus}.

\begin{figure}
	\includegraphics[width=0.6\textwidth]{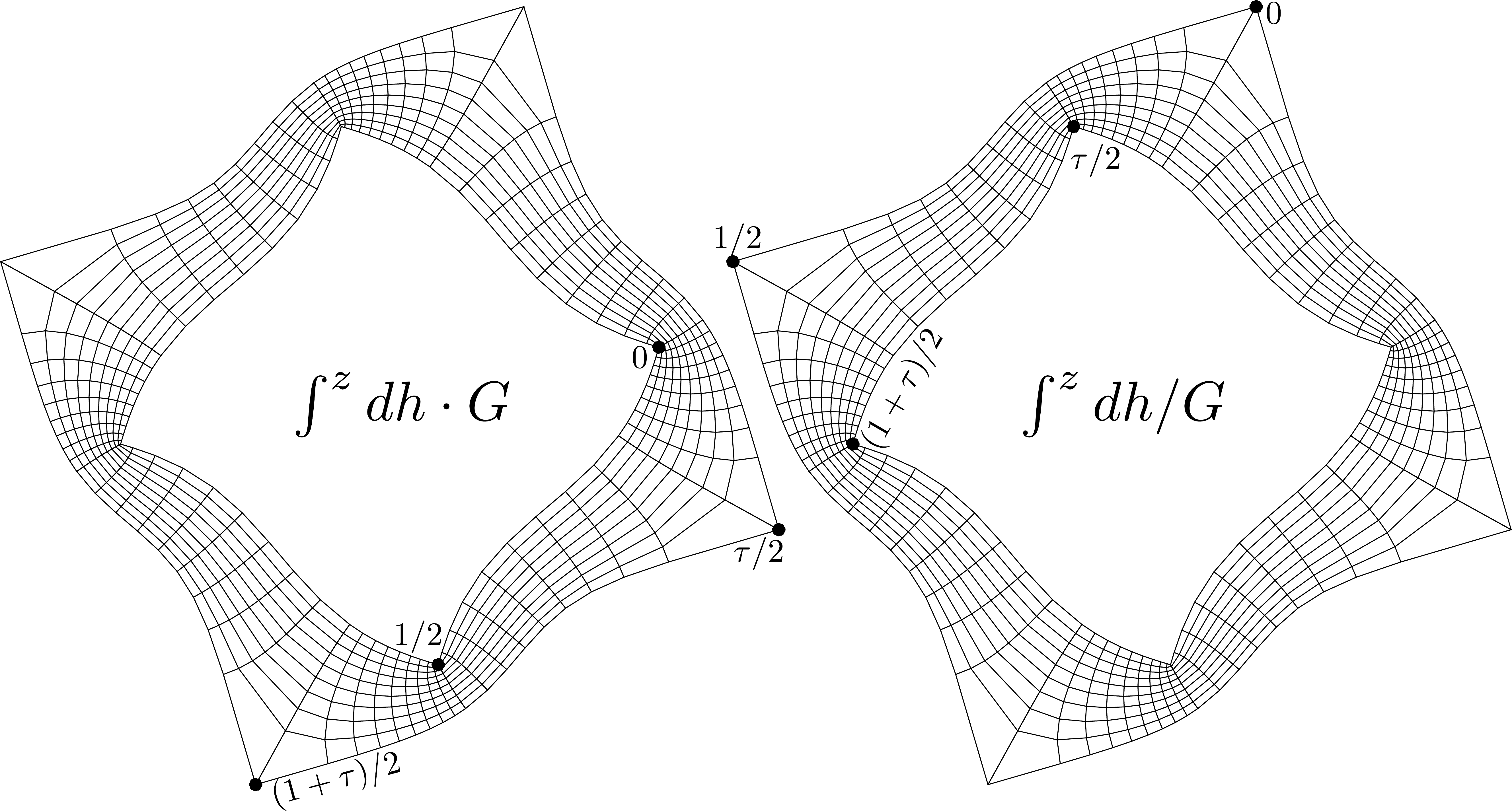}
	\caption{
		Plot of the lower annulus $0 < \im z < \im \tau/2$ under the maps $\Phi_1$
		and $\Phi_2$ for $\cT$, with $\tau = 0.3+0.2 i$.
	}
	\label{fig:Tannulus}
\end{figure}

\medskip

Combining the flat structures $\Phi_1$ and $\Phi_2$ gives us the image under
the Weierstrass parameterization.  Because of our choice of $dh$, we know that
the lines $\im z = 0$ and $\im z = \im \tau/2$ are mapped to two horizontal
planar curves, at heights $0$ and $\im \tau/2$ respectively.  The previous
analysis tells us that these are two congruent closed curves that look like
curved squares.  In particular, they admit rotational symmetry of order $4$.
The strip $0 < \im z < \im \tau/2$ is then mapped to a ``twisted square
catenoid'' bounded by these curves.  Moreover, the twisted catenoid admits
rotational symmetries of order 2 around horizontal axes that swap its boundaries.
See Figure~\ref{fig:Tcatenoid}.

\begin{figure}
	\includegraphics[width=0.25\textwidth,valign=t]{Figure/TCatenoid-10.pdf}
	\includegraphics[width=0.25\textwidth,valign=t]{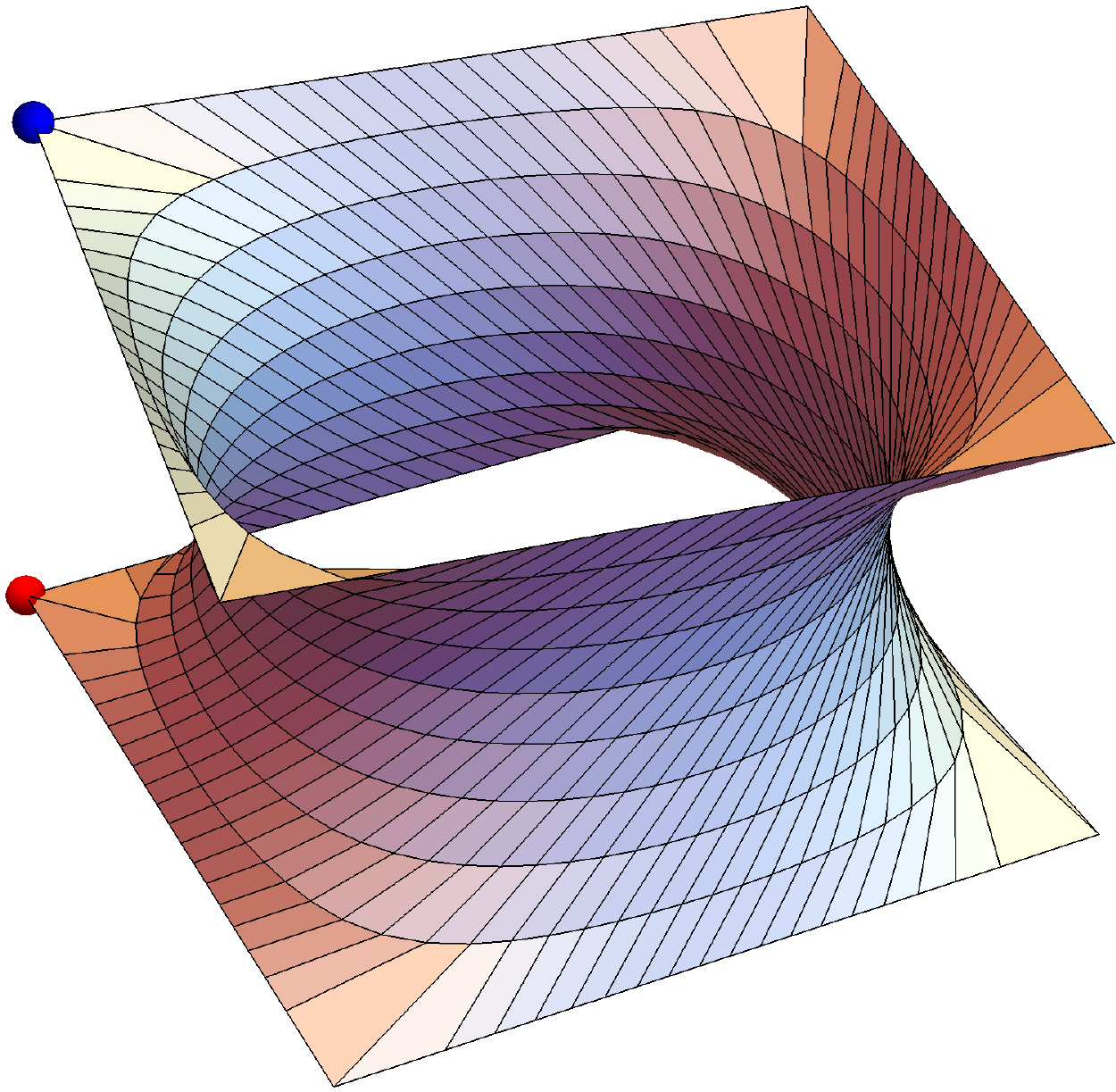}
	\includegraphics[width=0.25\textwidth,valign=t]{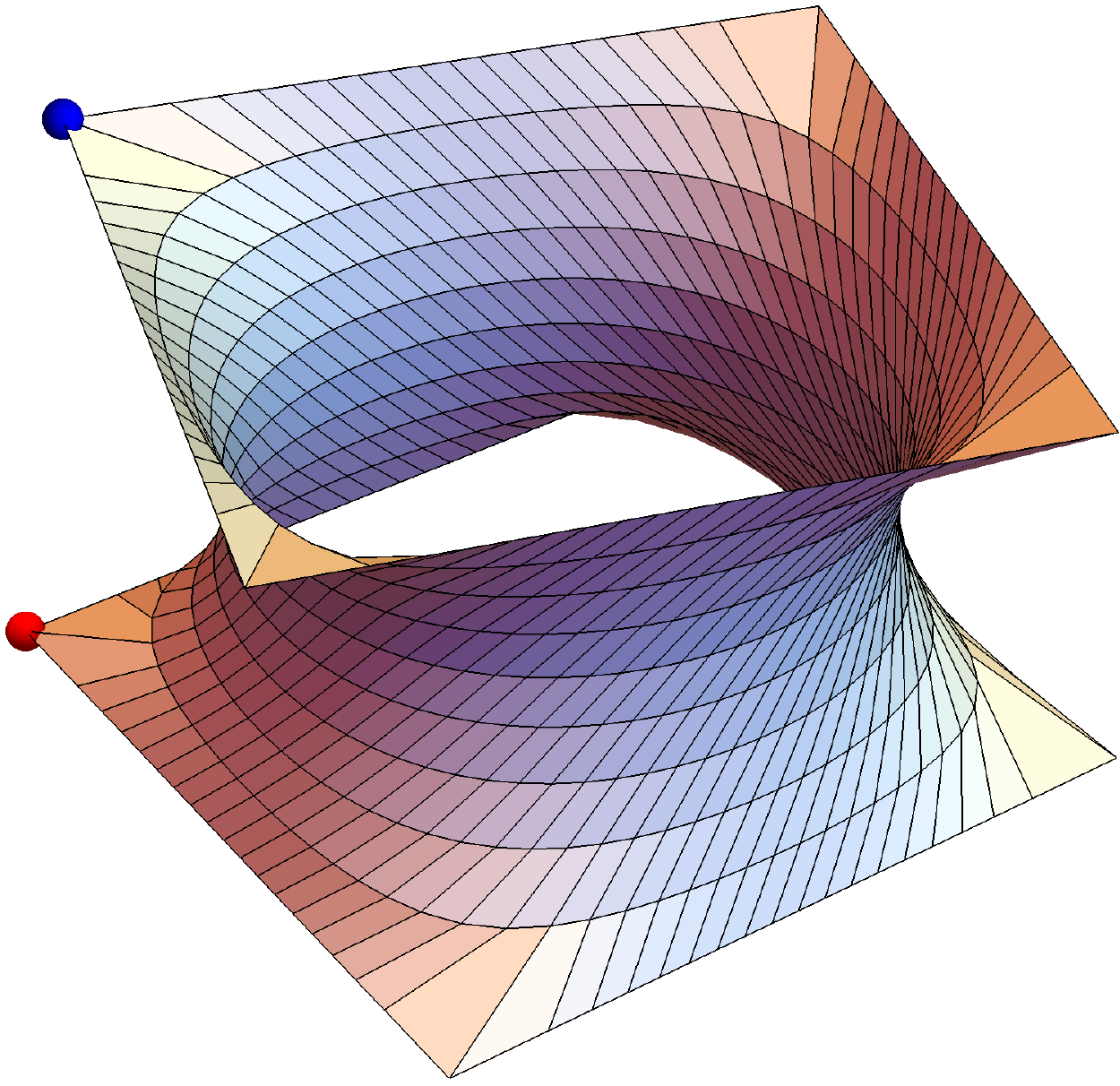}

	\includegraphics[width=0.25\textwidth,valign=t]{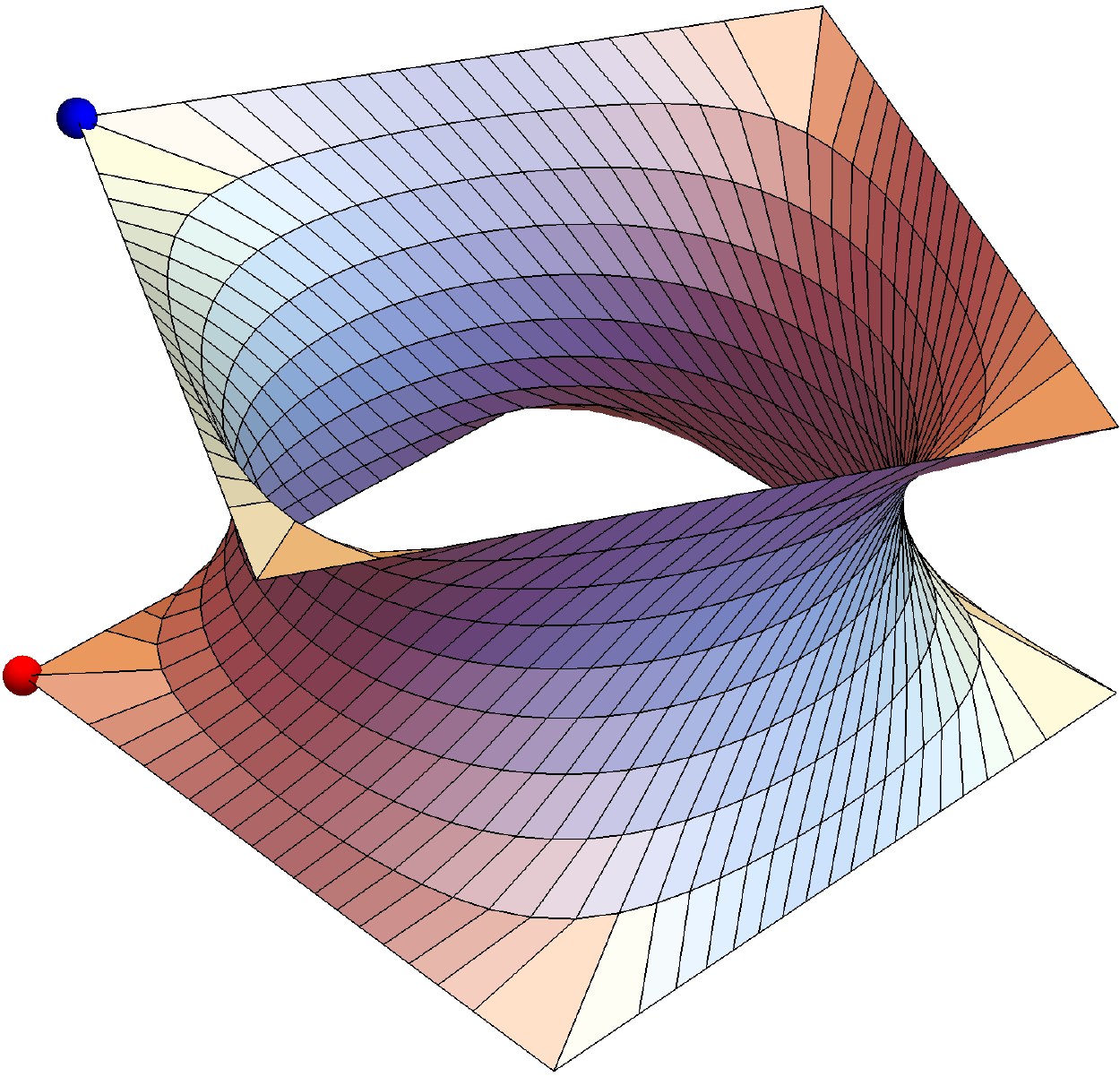}
	\includegraphics[width=0.25\textwidth,valign=t]{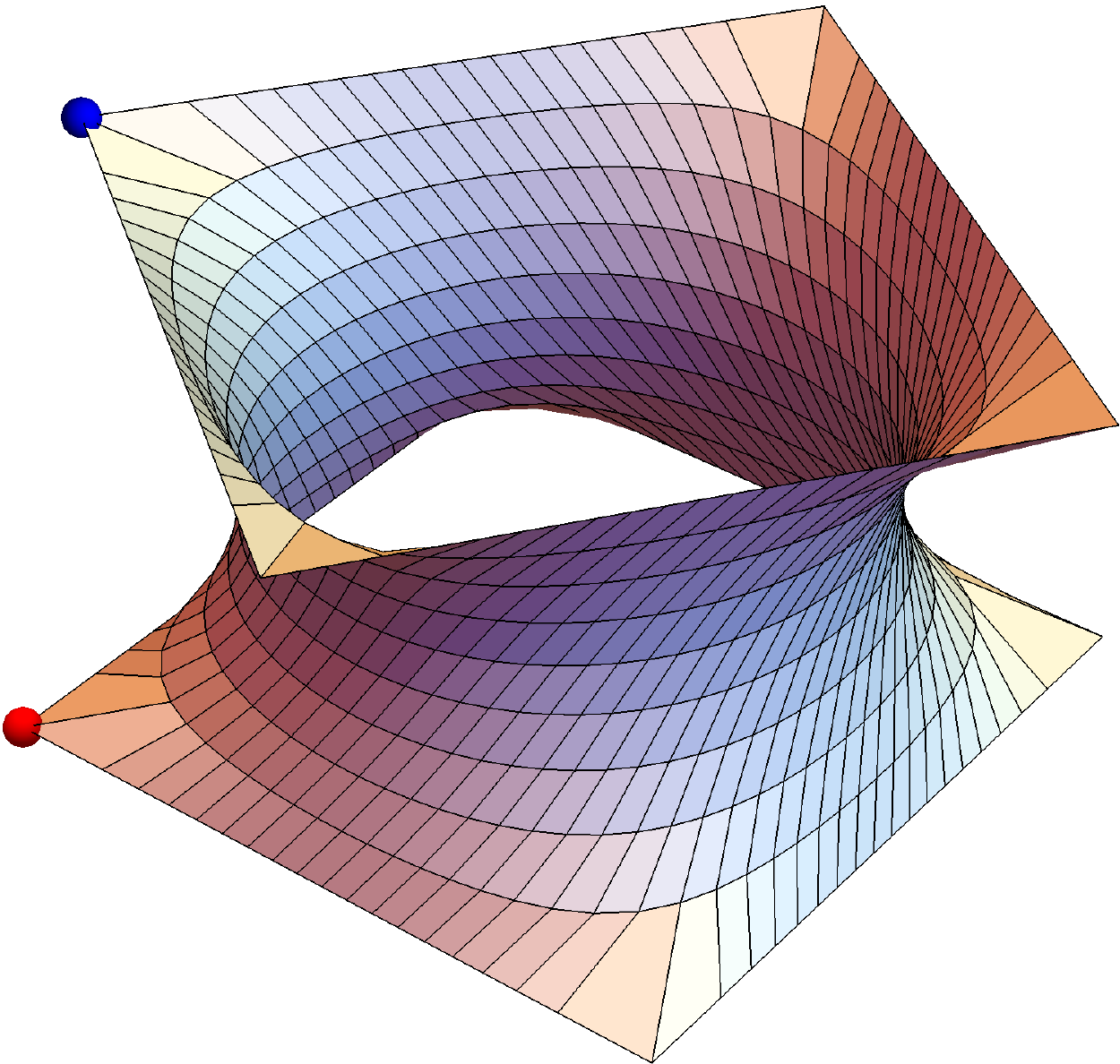}
	\includegraphics[width=0.25\textwidth,valign=t]{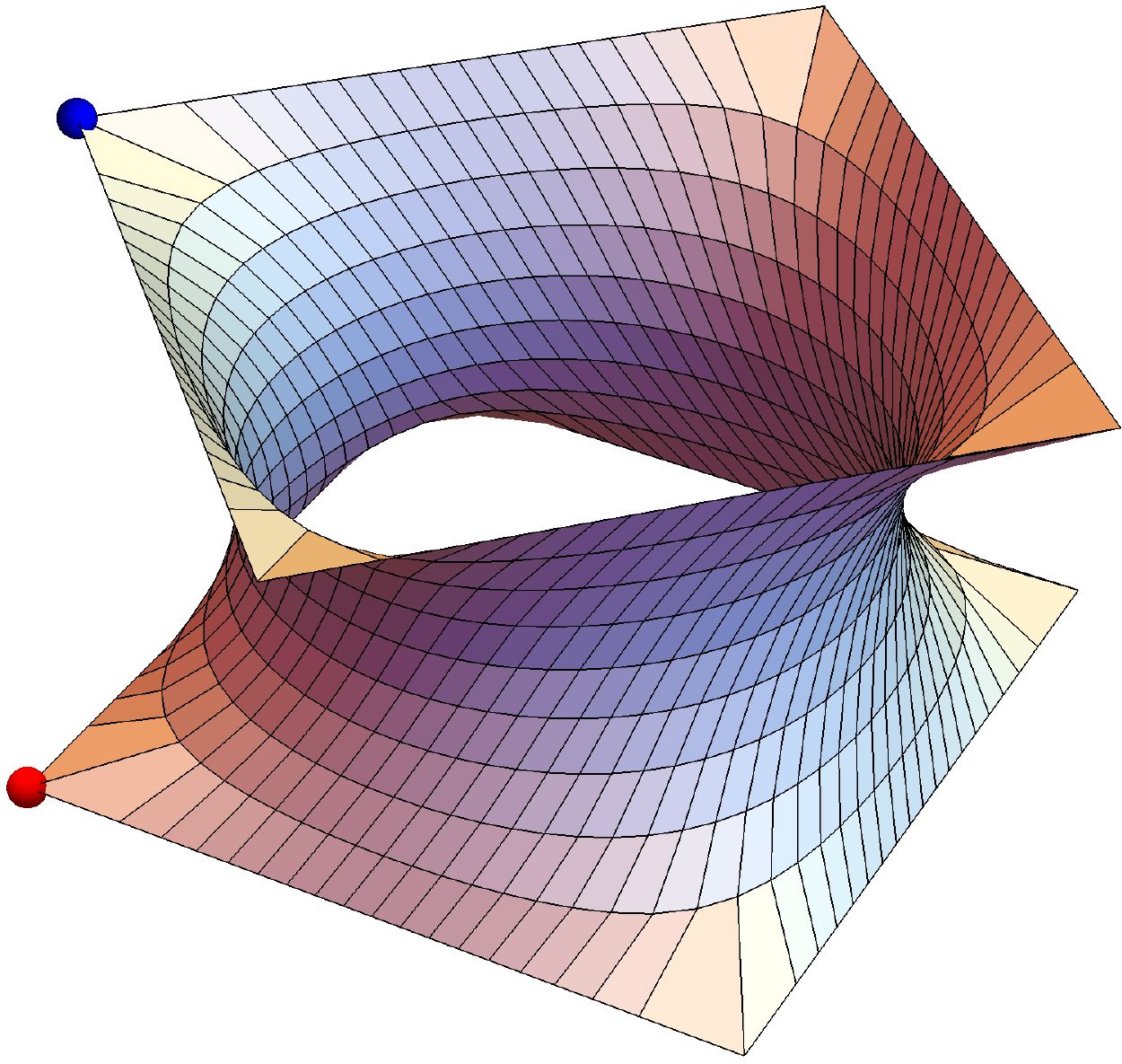}
	\caption{
		Plot of the annulus with $0 < \im z < \im \tau/2$ under the Weierstrass
		parameterization of $\cT$, for $\im\tau = 1$ and $\re\tau = -1, -0.9, -0.8,
		-0.7, -0.6, -0.5$, in this order.  The red and the blue points indicate the
		images of $0$ and $(1+\tau)/2$.  An increasing ``twist angle'' is visible.
		Note that bounding edges are in fact slightly curved, except for $\re\tau =
		-1$ and $-0.5$.
	}
	\label{fig:Tcatenoid}
\end{figure}

Inversion in the midpoint of a curved edge extends the surface with another
twisted square catenoid, which is the image of the strip $\im \tau/2 < \im z <
\im \tau$.  Repeated inversions in the midpoints of the curved edges extend the
catenoid infinitely into the space $\R^3$, but the result is usually not
embedded.

\medskip

The same argument applies to the Weierstrass data of $\cR$.  The lower annulus
lift to a strip in $\C$ of period $3/2$.  We then obtain a twisted triangular
catenoid with rotational symmetry of order~$3$, and repeated inversions in the
midpoints of the curved edges extend the catenoid infinitely into $\R^3$, but
usually not embedded.  See Figures~\ref{fig:Rannulus} and~\ref{fig:Rcatenoid}.

\medskip

\begin{figure}
	\includegraphics[width=0.6\textwidth]{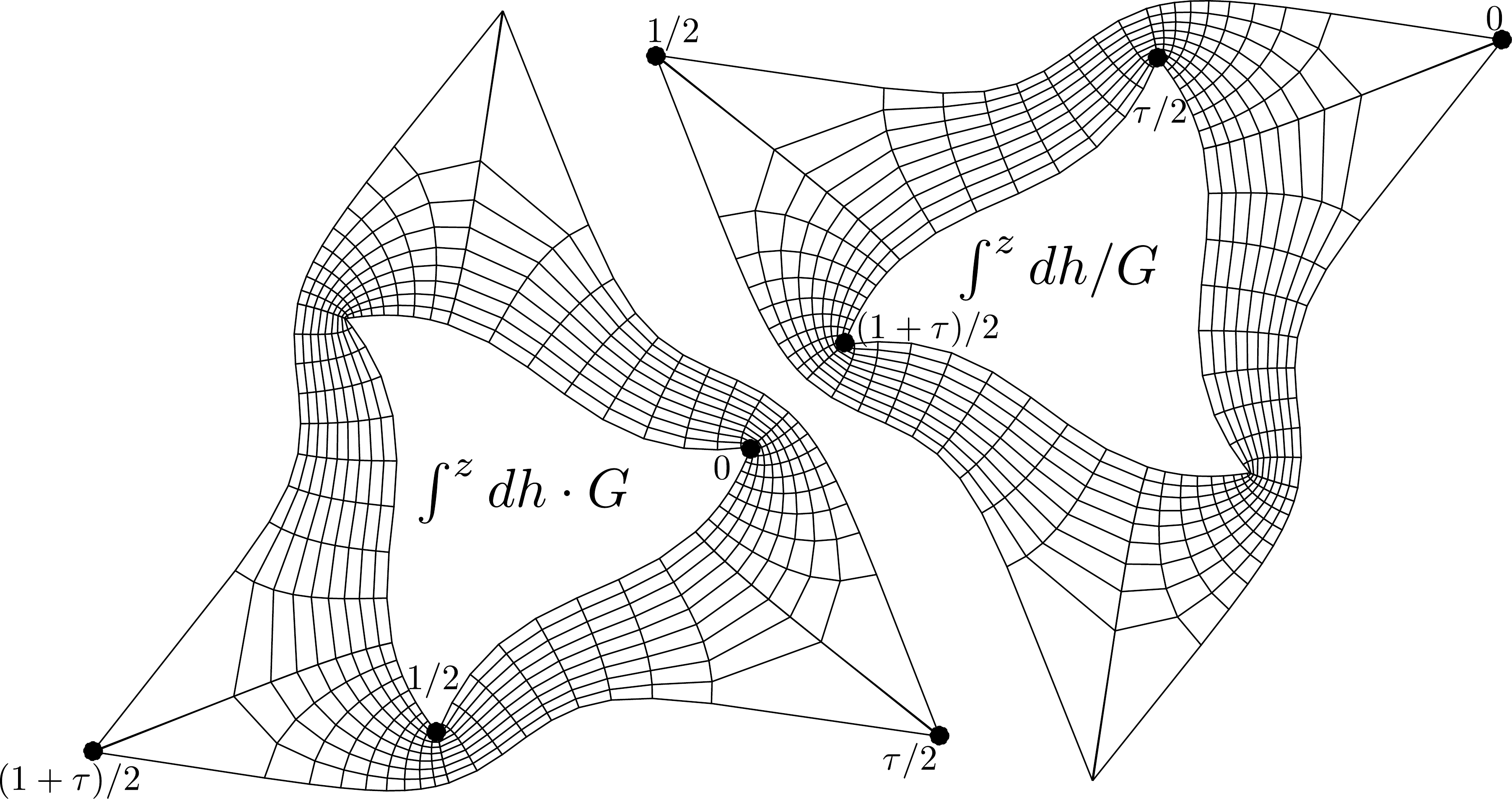}
	\caption{
		Plot of the annulus $0 < \im z < \im \tau/2$ under the maps $\Phi_1$ and
		$\Phi_2$ for $\cR$, with $\tau = 0.3+0.2 i$.
	}
	\label{fig:Rannulus}
\end{figure}

\begin{figure}
	\includegraphics[width=0.25\textwidth,valign=t]{Figure/RCatenoid-10.pdf}
	\includegraphics[width=0.25\textwidth,valign=t]{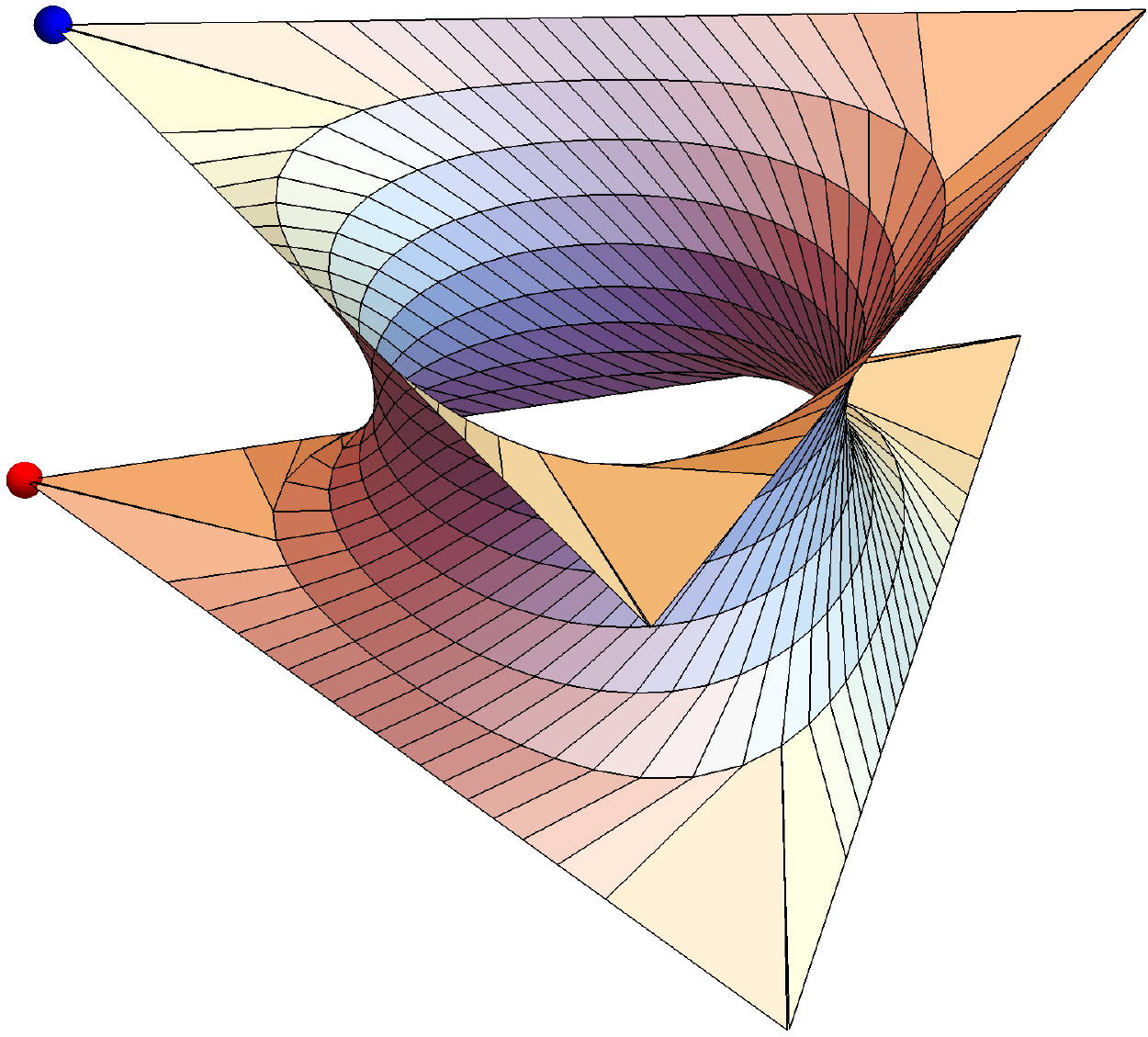}
	\includegraphics[width=0.25\textwidth,valign=t]{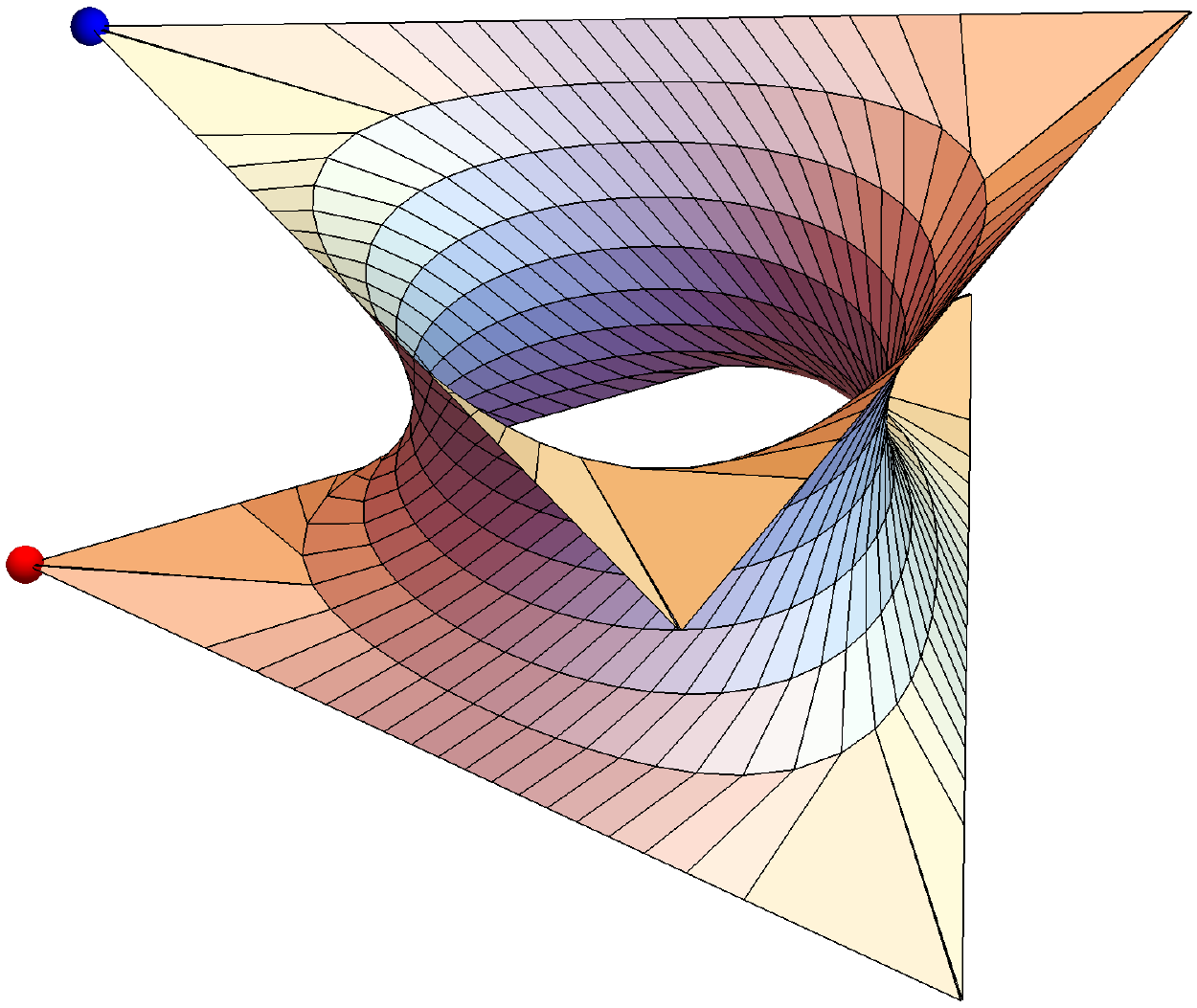}

	\includegraphics[width=0.25\textwidth,valign=t]{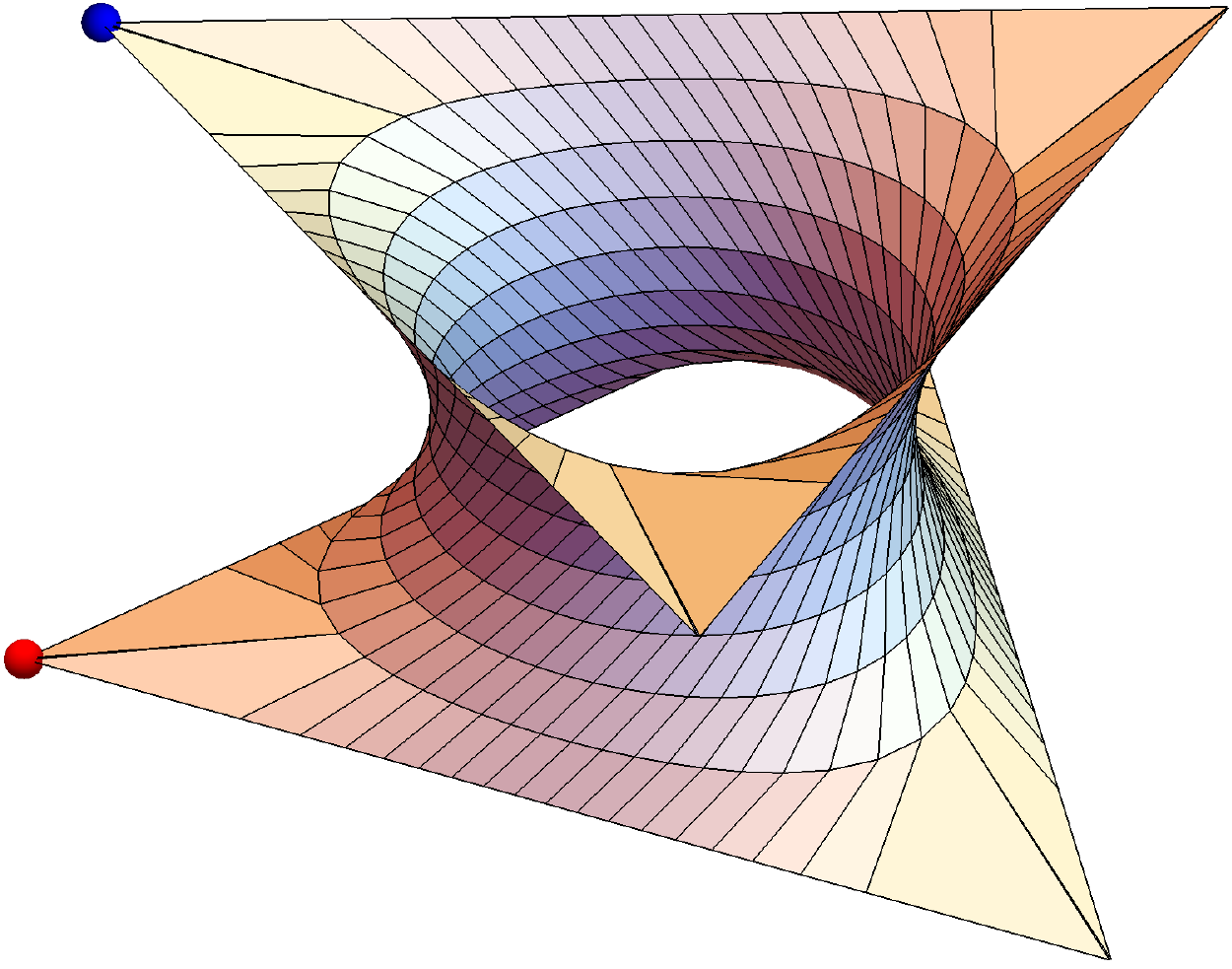}
	\includegraphics[width=0.25\textwidth,valign=t]{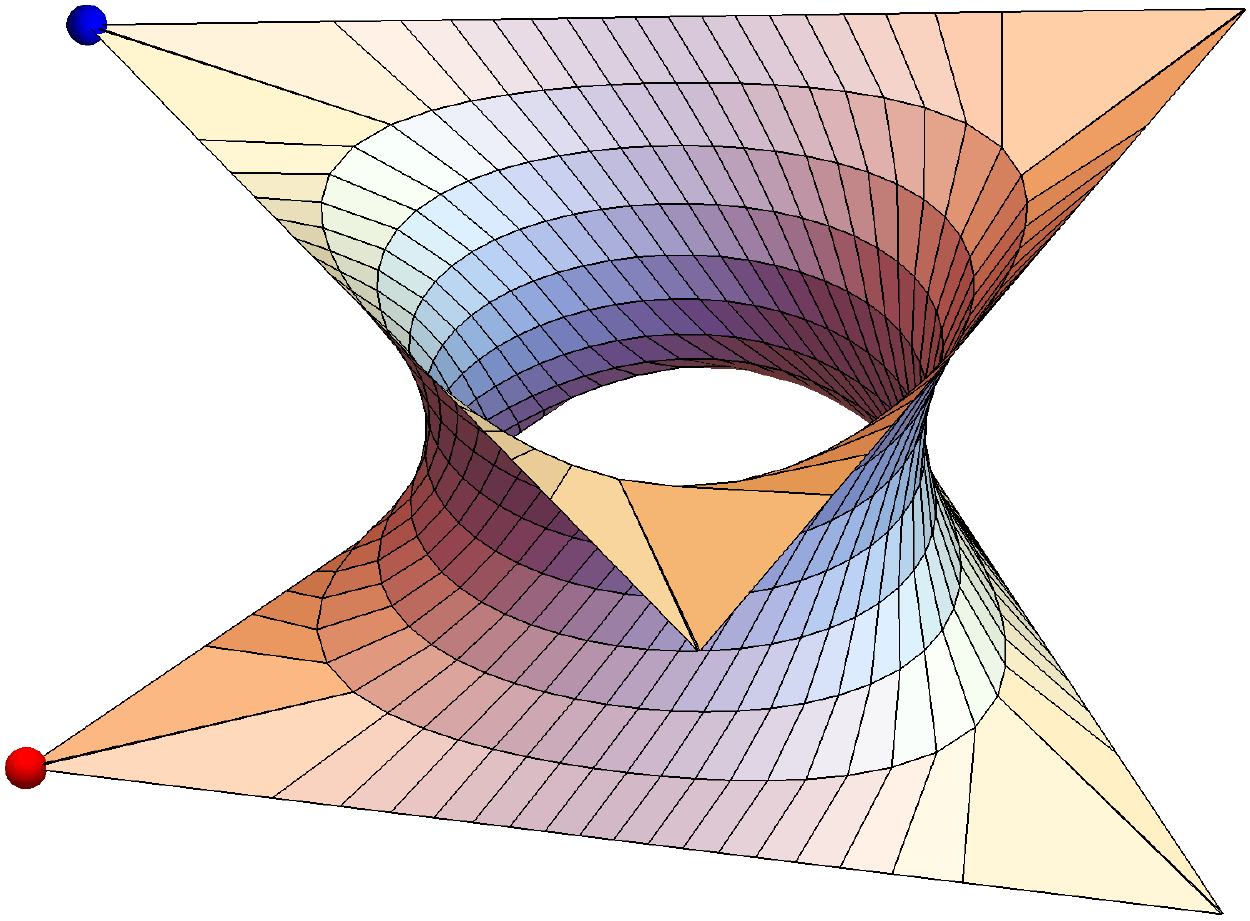}
	\includegraphics[width=0.25\textwidth,valign=t]{Figure/RCatenoid-05.pdf}
	\caption{
		Plot of the annulus with $0 < \im z < \im \tau/2$ under the Weierstrass
		parameterization of $\cR$, for $\im\tau = 1$ and $\re\tau = -1, -0.9, -0.8,
		-0.7, -0.6, -0.5$, in this order.  The red and the blue points indicate the
		images of $0$ and $(1+\tau)/2$.  An increasing ``twist angle'' is then
		visible.  Note that bounding edges are in fact slightly curved, except for
		$\re\tau = -1$ and $-0.5$.
	}
	\label{fig:Rcatenoid}
\end{figure}

It is interesting to observe the twisted catenoids when $\re\tau$ increases.

Let us start from a tP surface with $\re\tau = -1$.  Its catenoid is the
standard square catenoid, not twisted.  As we increase $\re\tau$, the square
catenoid becomes ``twisted'' in two senses: on the one hand, the bounding
squares become curved; on the other hand, horizontal projections of the squares
form an angle.  This ``twist angle'' seems to increase monotonically with
$\re\tau$ (we are not sure!); see Figure~\ref{fig:Tcatenoid}.  Remarkably, when
$\re\tau = -1/2$, reflectional symmetry is restored in the Weierstrass data,
and the catenoid is bounded by two straight squares forming a twist angle
of~$\pi/4$.

Then the catenoid is again ``twisted'' as we continue to increase $\re\tau$
until $1$.  During the process, the twist angle increases from $0$ at $\re\tau
= -1$ (tP), to $\pi/4$ at $\re\tau = -1/2$, to $\pi/2$ at $\re\tau = 0$
(gyroid), to $3\pi/4$ at $\re\tau = 1/2$, until $\pi$ at $\re\tau = 1$ (tD).
These are the only cases where reflectional symmetry is restored, and the
bounding edges are straight.  The term ``twist angle'' is in general
ill-defined\footnotemark, but carries a natural meaning in these cases.  Note
that, although the transform $\tau \mapsto \tau+1$ leaves the catenoid
invariant, the image of $(\tau+1)/2$ (blue point in Figure~\ref{fig:Tcatenoid})
is however rotated by $\pi/2$.  So the marked twisted catenoid is invariant
under the transform $\tau \mapsto \tau+4$.

\footnotetext{
	More precisely, I can think of several ways to define the ``twist angle''.
	They seem inconsistent, and it is not clear which is more beneficial.  Hence
	I prefer not to make the definition at the moment.
}

\begin{remark}
	The tG surfaces at $\re\tau = \pm 1/2$ deserve more attention, as the
	reflectional symmetries in their Gauss maps seem special.
\end{remark}

Similarly, for $\cR$, we can start from the triangular catenoid of H with
$\re\tau = -1$, and increase $\re\tau$ until $1/2$.  The bounding triangles
become curved and form a twisted angle that seems to increase monotonically
with $\re\tau$; see Figure~\ref{fig:Rcatenoid}.  In particular, the ``twist
angle'' increases from $0$ at $\re\tau = -1$ (H), to $\pi/3$ at $\re\tau =
-1/2$ (gyroid), to $2\pi/3$ at $\re\tau = 0$ (Lidinoid), until $\pi$ at
$\re\tau = 1/2$ (rPD).  These are the only cases where the reflectional
symmetry is restored, the bounding edges are straight, and the meaning of the
``twist angle'' is clear.  The marked twisted catenoid is invariant under the
transform $\tau \mapsto \tau+3$.

\subsection{Vertical and horizontal associate angles}

The Weierstrass parameterization defines an immersion only if the period
problems are solved.  That is, the integrations around closed curves on
$\Sigma$ should all vanish (up to $\Lambda$).  The period problems for Schwarz'
surfaces were explicitly solved in~\cite{weyhaupt2006}.

Gro\ss e--Brauckmann and Wohlgemuth~\cite{kgb1996} proposed a convenient way to
visualize the gyroid, which we now generalize to all surfaces in $\cT$ and
$\cR$.  This will reduce the period problems to just one equation. 

\medskip

We have demonstrated a twisted catenoid in the associate family of every
surface in $\cT$ or $\cR$.  As one travels from the twisted catenoids along the
associate family, the bounding twisted squares or triangles become ``twisted
helices'', and the twisted catenoid opens up into a ribbon bounded by these
helices.

For an embedded TPMSg3 $M/\Lambda$, Meeks~\cite{meeks1990} proved that its
hyperelliptic points can be identified to the lattice and half-lattice points
of $\Lambda$.  Then the symmetries of $\cT$ and $\cR$ imply that the poles
and zeros of the Gauss map are
\begin{itemize}
	\item aligned along vertical lines arranged in a square or triangular lattice and
	\item alternatingly arranged and equally spaced on each vertical line.
\end{itemize}
Conversely, if these are the cases, then the two ribbons forming the
fundamental unit ``fit exactly into each other''.  That is, their boundaries
are identified in the same pattern as the twisted catenoids (see
Figures~\ref{fig:Ttorus} and~\ref{fig:Rtorus}).  This is guaranteed by the
screw and the inversional symmetries.  So the properties listed above form a
sufficient condition for the immersion.

\medskip

To be more precise, we define the pitch of a helix to be the increase of height
after the helix makes a full turn.  Then the poles and zeros are alternatingly
arranged and equally spaced on each vertical line if the pitch of each helix is
an even multiple of the minimum vertical distance between the poles and the
zeros.  In particular, for the gyroid, the pitch doubles the minimum vertical
distance.  So we expect the same property for tG and rGL surfaces.  

\begin{remark}
	The ratio of the pitch over the minimum vertical distance can also take other
	values.  They will be discussed in Section~\ref{sec:discuss}.
\end{remark}

For tG surfaces, this means that the integral of the height differential $dh$
from $0$ to $2$ doubles the integral from $0$ to $(1+\tau)/2$.  Or
equivalently, the integral from $0$ to $1$ equals the integral from $0$ to
$(1+\tau)/2$.  This can be easily achieved by adjusting the associate angle
$\theta$ to (compare~\cite[Definition~4.2]{weyhaupt2008})
\[
	\theta = \theta_v(\tau) = \arg(\tau-1) - \pi/2 \quad \text{for tG.}
\]
Similarly, for rGL surfaces, this means that the integral of $dh$ from $0$ to
$3/4$ equals the integral from $0$ to $(1+\tau)/2$.  Then we deduce that
\[
	\theta = \theta_v(\tau) = \arg(\tau-1/2) - \pi/2 \quad \text{for rGL.}
\]

\medskip

We now calculate the associate angle in another way, using the fact that the
images of $0$ and $(1+\tau)/2$ are vertically aligned, i.e.\ have the same
horizontal coordinates.

First note that, by the symmetry $\sn(2K+iK'-z;\tau) = 1/\sn(z;\tau)$, we
have the identity
\[
	\int_0^{(1+\tau)/2} dz/G = \int_0^{(1+\tau)/2} dz \cdot G =: \psi(\tau)
\]
We may place the image of $0$ at the origin.  First look at the surface with
$\theta=0$, hence $dh = dz$.  Then the horizontal coordinates of the image of
$(1+\tau)/2$ are
\[
	\re \int_0^{(1+\tau)/2} \Big(\frac{1}{2}\big(\frac{1}{G}-G\big), \frac{i}{2}\big(\frac{1}{G}+G\big)\Big)dz = (0, -\im \psi(\tau)).
\]
Then we look at the surface with $\theta=\pi/2$, hence $dh = e^{-i\pi/2} dz =
-i dz$ (the conjugate surface).  Then the coordinates are
\[
	\re \int_0^{(1+\tau)/2} \Big(-\frac{i}{2}\big(\frac{1}{G}-G\big), \frac{1}{2}\big(\frac{1}{G}+G\big)\Big)dz = (0, \re \psi(\tau)).
\]
So for the surface with associate angle $\theta$, the first coordinate is
always $0$, while the second coordinate
\[
	- \cos\theta \im\psi(\tau) + \sin\theta \re\psi(\tau)
\]
vanishes when
\[
	\theta = \theta_h(\tau) := \arg\psi(\tau).
\]

\medskip

We have shown that

\begin{lemma}
	The Weierstrass data given by Lemmas~\ref{lem:dh} and~\ref{lem:G} define an
	immersion if and only if
	\[
		\theta_h(\tau)=\theta_v(\tau),
	\]
	or more explicitly,
	\begin{equation}\label{eq:period}
		\arg\psi(\tau) =
		\begin{cases}
			\arg(\tau-1)-\pi/2 & \text{for $\cT$;}\\
			\arg(\tau-1/2)-\pi/2 & \text{for $\cR$.}
		\end{cases}
	\end{equation}
\end{lemma}

We are finally ready to give the existence proof.

\section{Existence proof}\label{sec:proof}

\subsection{tG family}

In Figure~\ref{fig:tGtau}, we show for $\cT$ the numerical solutions
to~\eqref{eq:period} with $-1<\re \tau < 1$, accompanied by two half-circles
representing the CLP family.  Our task is to prove the existence of the
continuous 1-parameter solution curve that we see in the picture, which we call
the tG family.  Let the shaded domain in the figure be denoted by
\[
	\Omega_t := \{ \tau \mid \im\tau>0, -1<\re\tau<1, |\tau\pm1/2|>1/2\}.
\]

\begin{figure}
	\includegraphics[width=0.5\textwidth]{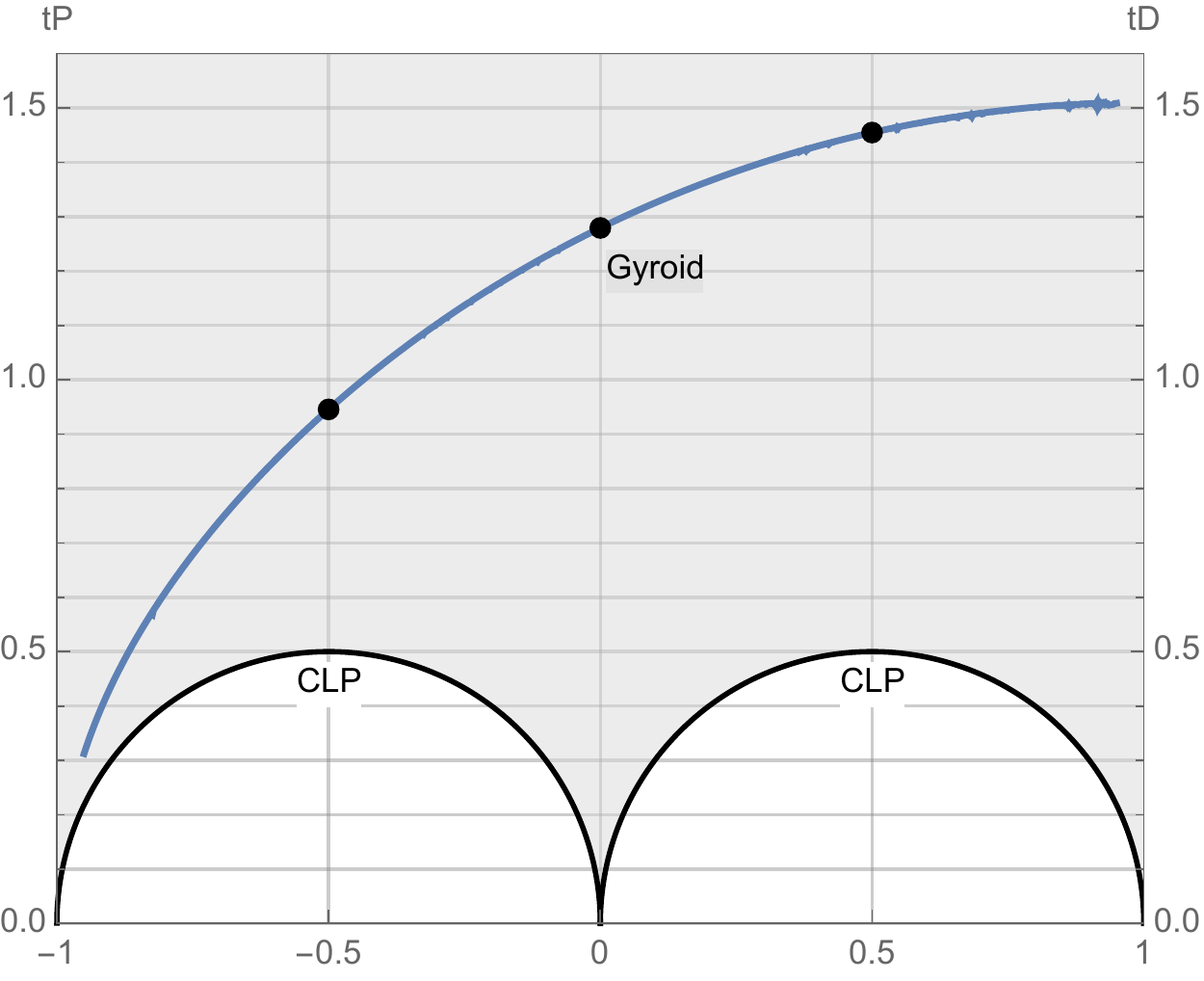}
	\caption{
		Solutions for $\cT$ to~\eqref{eq:period} with $-1 < \re \tau < 1$.
	}
	\label{fig:tGtau}
\end{figure}

\begin{proposition}
	There exists a continuous 1-dimensional curve of $\tau$ in $\Omega_t$ that
	solves~\eqref{eq:period}.  This curve tends to $-1$ at one end and to $1+t i$
	for some $0 < t < \infty$ at the other end.  Moreover, the TPMSg3s
	represented by points on the curve are all embedded.
\end{proposition}

\begin{proof}
	We examine the angles $\theta_v$ and $\theta_h$ on the boundaries of
	$\Omega_t$.

	On the vertical line $\re\tau = -1$, we see immediately that $0 < \theta_v <
	\pi/2$.  Since this line corresponds to the tP family, we know very well
	that the image of $(1+\tau)/2$ is directly above the image of $0$ when $dh =
	e^{-i \pi/2} dz$.  Hence $\theta_h = \pi/2 > \theta_v$.

	The half-circle $|\tau+1/2| = 1/2$ corresponds to the CLP family.  We know
	very well that the image of $(1+\tau)/2$ is directly above the image of $0$
	when $dh = \exp(-i(\arg \tau - \pi/2)) dz$, so $\theta_h = \arg\tau-\pi/2$.
	Then the inequality $\theta_h < \theta_v$ follows from elementary geometry.

	On the half-circle $|\tau-1/2| = 1/2$, it follows from elementary geometry
	that $\theta_v = \arg\tau$.  This half-circle corresponds again to the CLP
	family.  When $dh = \exp(-i \theta_v) dz$, the image of $\tau/2$ is directly
	above the image of $0$.  In this case, we know very well that the flat
	structure of $\Phi_1$ is as depicted in Figure~\ref{fig:CLP};
	see~\cite{weyhaupt2006}.  In particular, we have
	\[
		\arg \int_0^{(1+\tau)/2} dh \cdot G = \arg \int_0^{(1+\tau)/2} dz \cdot G \exp(-i \theta_v) = \theta_h - \theta_v < 0,
	\]
	hence $\theta_h < \theta_v$.

	\begin{figure}
		\includegraphics[width=0.3\textwidth]{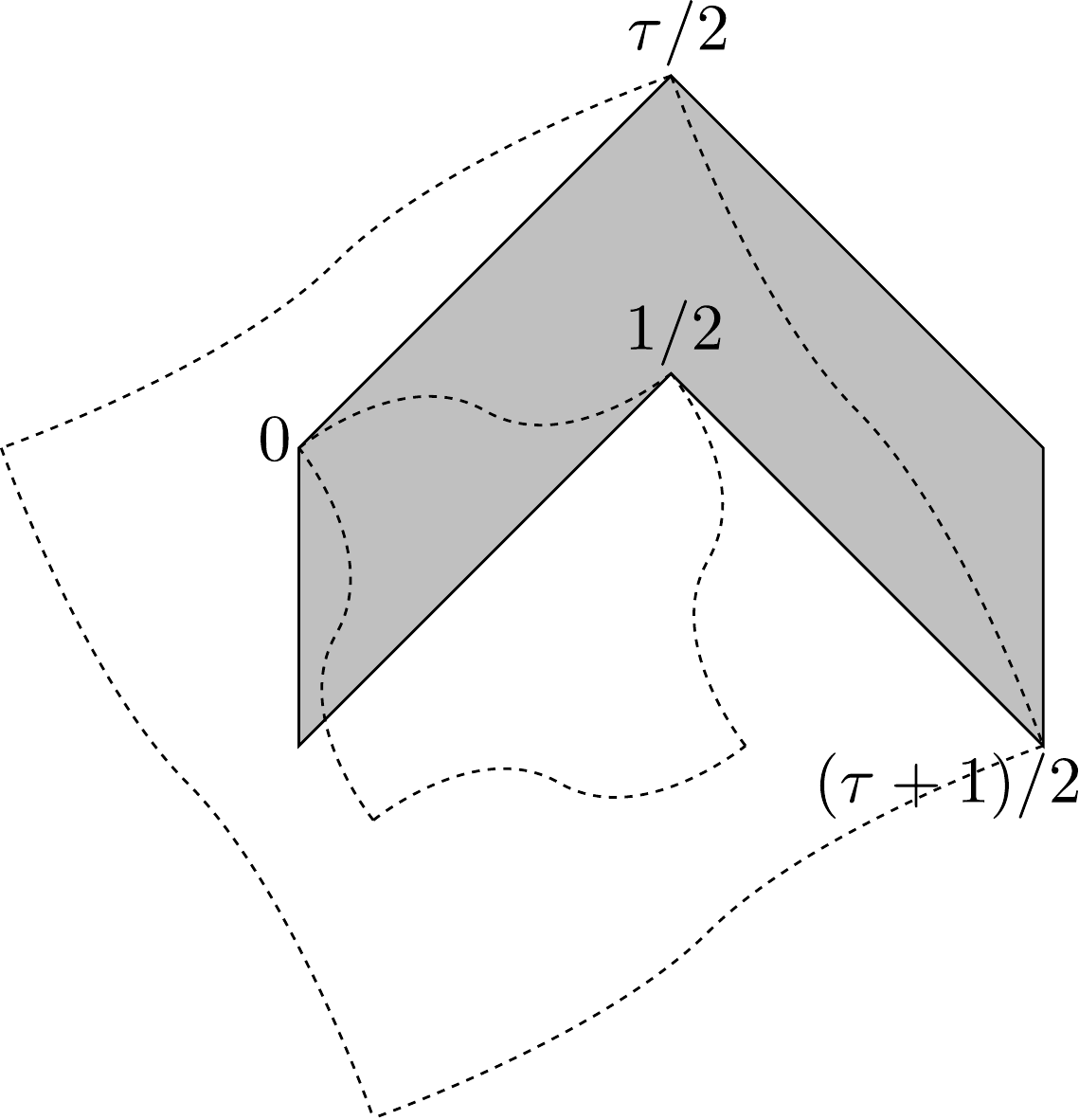}
		\caption{
			Image of the rectangle with vertices at $0$, $1/2-\tau/2$, $1/2+\tau/2$,
			$\tau$, under the map $\Phi_1$ for a CLP surface (the grey area).  The
			twisted square annulus is sketched (exaggeratedly) in dashed curves.
			\label{fig:CLP}
		}
	\end{figure}

	On the vertical line $\re\tau = 1$, we see immediately that $\theta_v = 0$.
	Since this line corresponds to the tD family, we know very well that the
	image of $(1+\tau)/2$ is directly above the image of $0$ when $dh = dz$, so
	$\theta_h = 0$.  Hence this line solves the period condition $\theta_h =
	\theta_v$, but not very helpful for our proof.

	As $\im\tau \to \infty$, we see immediately that $\theta_v \to 0$.
	Asymptotic behavior of $\theta_h$ is technical, so we postpone the details to
	Appendix, where Lemma~\ref{lem:tauinf} states that $\theta_h \to
	\re(1-\tau)\pi/4$.  Hence $\theta_h > \theta_v$ for $\im \tau$ sufficiently
	large.

	Now consider $\tau = 1 - \epsilon + i \eta$ for small $\epsilon$.  It is
	immediate that the derivative of $\theta_v$ with respect to $\epsilon$ at
	$\epsilon = 0$ is $1/\eta$, hence tends to $0$ as $\eta \to \infty$.
	Meanwhile, the asymptotic behavior of $\theta_h$ tells us that $\partial
	\theta_h / \partial\epsilon \to \pi/4 > 0$ as $\eta \to \infty$.
	Consequently, there exist two positive numbers $H$ and $\delta$ such that
	$\theta_h(1 - \epsilon + \eta i) > \theta_v(1 - \epsilon + \eta i)$ for all
	$\eta>H$ and $\epsilon<\delta$.

	As $\tau \to 1$ within $\Omega_t$, Lemma~\ref{lem:tau0} in Appendix claims
	that $\theta_h(\tau) \to 0$ from the negative side, but one sees immediately
	that $\theta_v(\tau) \to 0$ from the positive side.  Consequently, there
	exists a neighborhood $U$ of $1$ such that $\theta_h(\tau) < \theta_v(\tau)$
	for all $\tau \in U \cap \Omega_t$.

	\medskip

	Note that $\theta_h$ and $\theta_v$ are both real analytic functions in the
	real and imaginary part of $\tau$, hence the solution set of the period
	condition~\eqref{eq:period} is an analytic set.  By the continuity, we
	conclude that the solution set contains a connected component that separates
	the half-circles $|\tau \pm 1/2| = 1/2$ from the line $\re \tau = -1$ and the
	infinity.  Moreover, this set must also separate a neighborhood $U$ of $1$
	from the set $\{\tau \mid \re \tau > 1-\delta, \im \tau > H\}$.  Because of
	the analyticity, we may extract a continuous curve from the connected
	component, which is the tG family.  In particular, this curve must tend to
	the common limit of CLP and tP (a saddle tower of order 4 at $\tau \to -1$),
	and intersect the tD family at a finite, positive point.

	The proof of~\cite[Lemma~4.4]{weyhaupt2008} implies that the gyroid is the
	only embedded TPMSg3 on the vertical line $\re \tau = 0$ that solves the
	period condition.  Hence the tG family must contain the gyroid, whose
	embeddedness then ensures the embeddedness of all TPMSg3s in the tG family.
	This follows from~\cite[Proposition~5.6]{weyhaupt2006}, which was essentially
	proved in~\cite{meeks1990}.
\end{proof}

In Figure~\ref{fig:tGgallery}, we show two adjacent ribbons for some tG
surfaces.  They form a fundamental unit for the translational symmetry group.

\begin{figure}
	\includegraphics[width=0.6\textwidth]{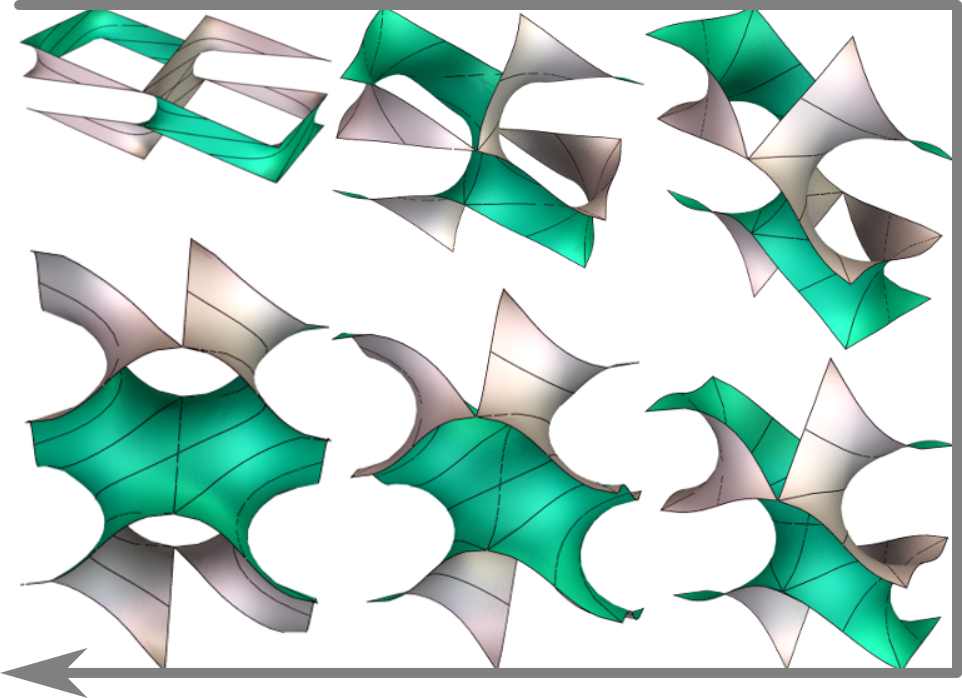}
	\caption{
		tG surfaces with $\re\tau=-0.95$, $-0.6$, $-0.2$, $0.2$, $0.6$ and $0.95$,
		in the indicated order.
	}
	\label{fig:tGgallery}
\end{figure}

\subsection{rGL family}

In picture~\ref{fig:tGtau}, we show for $\cR$ the numerical solutions
to~\eqref{eq:period} with $-1<\re \tau < 1/2$.  The two half-circle represents
an order-3 analogue of the CLP surface, termed hCLP in~\cite{lidin1990}, but
already known to Schwarz~\cite{schwarz1890}; see also~\cite{fogden1992,
ejiri2015} .  Surfaces in hCLP are not embedded, but also not dense in the
space.  It is very easy to visualize and behaves very much like CLP.  In
particular, it's Weierstrass data is as shown in Figure~\ref{fig:hclp}; compare
CLP in Table~\ref{tbl:weierstrass}.

Our task is to prove the existence of the continuous 1-parameter solution curve
that we see in the picture, which we call the rGL family.  Let the shaded
domain in the figure be denoted by
\[
	\Omega_r = \{\tau \mid \im\tau>0, -1<\re\tau<1/2, |\tau\pm 1/2|>1/2\}.
\]

\begin{figure}
	\includegraphics[width=0.5\textwidth]{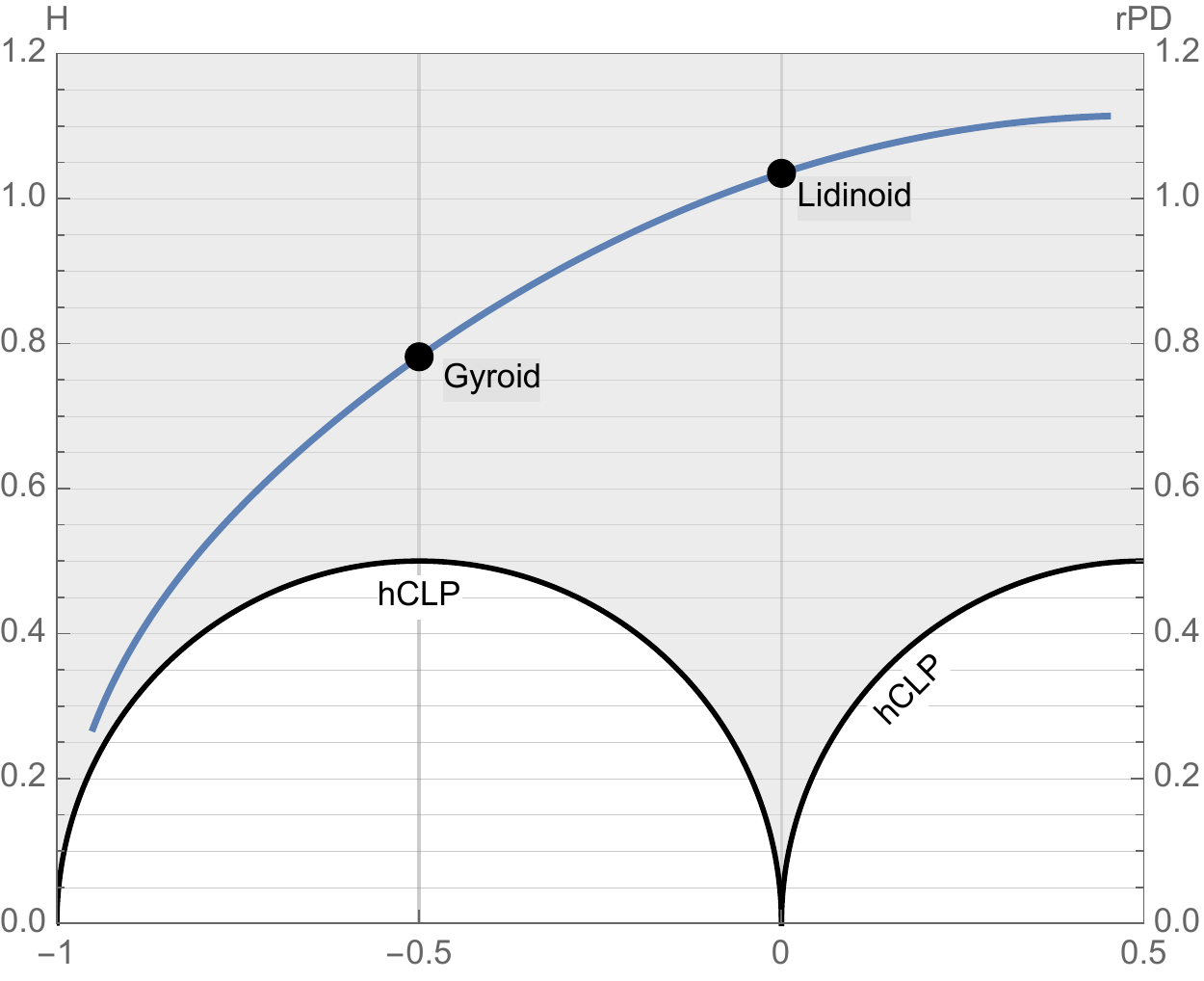}
	\caption{Solutions for $\cR$ to~\eqref{eq:period} with $-1 < \re \tau < 1/2$.}
	\label{fig:rGLtau}
\end{figure}

\begin{figure}
	\tikz[baseline=(current bounding box.center),scale = 2]{
		\fill[fill=lightgray] (0,0)--(-0.36,0.48)--(0.14,0.48)--(0.5,0)--cycle;
		\filldraw (0,0) circle [radius = .05];
		\filldraw (0.5,0) circle [radius = .05];
		\filldraw[fill=white] (-0.18,0.24) circle [radius = .05];
		\filldraw[fill=white] (0.32,0.24) circle [radius = .05];
		\draw[very thick,->] (0,0)--(0.24,0.18);
		\draw[dashed] (0,0) arc (0:180:0.5);
	}
	\caption{Weierstrass data for hCLP.}
	\label{fig:hclp}
\end{figure}

\begin{proposition}
	There exists a continuous 1-dimensional curve of $\tau$ in $\Omega_r$ that
	solves~\eqref{eq:period}.  This curve tends to $-1$ at one end and to $1/2+t
	i$ for some $0 < t < \infty$ at the other end.  Moreover, the triply periodic
	minimal surfaces of genus three represented by points on the curve are all
	embedded.
\end{proposition}

\begin{proof}
	Part of the proof is very similar to tG, so we just provide a sketch.

	The line $\re\tau = -1$ corresponds to the H family, and we have $\theta_h =
	\pi/2 > \theta_v$.

	The half-circle $|\tau+1/2| = 1/2$ corresponds to the hCLP family, and we
	have $\theta_h < \theta_v$.

	The line $\re\tau = 1/2$ correspond to the rPD family, and we have $\theta_h
	= \theta_v = 0$.  This solves the period condition, but not helpful for us.

	As $\im\tau \to \infty$, we have $\theta_h \to \re(1/2-\tau)\pi/3 > \theta_v
	= 0$.  The argument for the asymptotic behavior is very similar as in the
	proof of Lemma~\ref{lem:tauinf}, so we will not repeat it.

	By the same argument as for tG, we conclude that $\theta_h > \theta_v$ for
	$\tau = 1/2-\epsilon+i\eta$ as long as $\eta > H$ and $\epsilon < \delta$ for
	some positive constants $H$ and $\delta$.

	\medskip

	More care is however needed on the half-circle $|\tau-1/2| = 1/2$.  It
	corresponds again to the hCLP family.  When $dh = \exp(-i \arg\tau) dz$, the
	image of $\tau/2$ is directly above the image of $0$.  In this case, we know
	very well that the flat structure of $\Phi_1$ is as depicted in
	Figure~\ref{fig:hCLP}.  We see that
	\[
		\arg \int_0^{(1+\tau)/2} dh \cdot G = \arg \int_0^{(1+\tau)/2} dz \cdot G \exp(-i \arg\tau) = \theta_h - \arg\tau< 0.
	\]
	\begin{figure}[hbt]
		\includegraphics[width=0.3\textwidth]{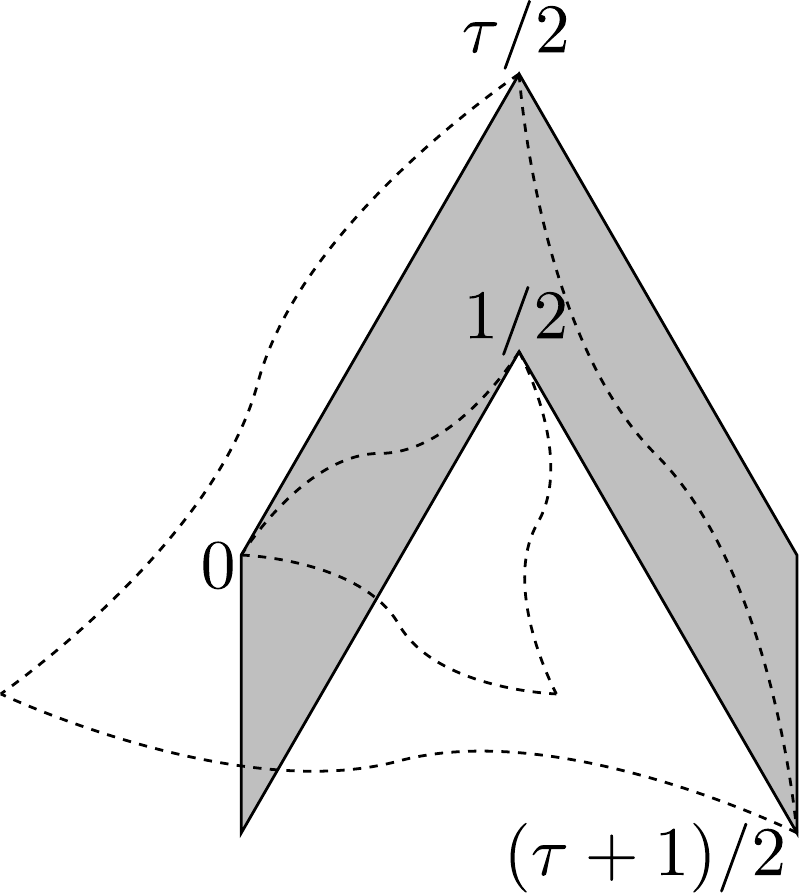}
		\caption{
			Image of the rectangle with vertices at $0$, $1/2-\tau/2$, $1/2+\tau/2$,
			$\tau$, under the map $\Phi_1$ for an hCLP surface (the grey area).  The
			twisted triangular annulus is sketched (exaggeratedly) in dashed curves.
			\label{fig:hCLP}
		}
	\end{figure}

	Let $a$ denote the length of the tilted segments (e.g.\ from $0$ to $\tau/2$)
	in the flat structure, and $b$ the length of the vertical segments (e.g.\
	from $1/2$ to $\tau/2$).  If follows from an extremal length argument that
	the ratio $a/b$ increases monotonically as $\tau$ travels along the
	half-circle with increasing $\re\tau$.  Then we see that $\theta_h-\arg\tau$
	increases monotonically from $-\pi/2$ to $0$.

	On the other hand, it follows from elementary geometry that $\theta_v -
	\arg\tau = \arg\tau - \pi/2$, which decreases monotonically as $\tau$ travels
	along the half-circle with increasing $\re\tau$.  Consequently, there is a
	unique $\tau$ on the half-circle for which $\theta_v = \theta_h$, and we know
	very well that this occurs at $\tau = (1+i)/2$.  At this point, it is interesting
	to verify that $a=b$, hence $\theta_h - \arg\tau = -\pi/4 = -\arg\tau$.

	Therefore, by monotonicity, we have $\theta_h < \theta_v$ on the left quarter
	of this half-circle.

	\medskip

	We then conclude the existence of a continuous curve of $\tau$, namely the rG
	family, that solves $\theta_h = \theta_v$ and separates the half-circles
	$|\tau \pm 1/2| = 1/2$ from the line $\re \tau = -1$ and the infinity.  This
	curve must tend to the common limit of hCLP and H (a saddle tower of order 3
	at $\tau \to -1$) and intersect rPD at a finite point.

	Moreover, the Lidinoid and the gyroid are the unique solutions on their
	respective vertical lines.  Hence the rGL family must contain them both.
	Their embeddedness then ensures the embeddedness of all TPMSg3s in the rGL
	family.
\end{proof}

In Figure~\ref{fig:rGLgallery}, we show two adjacent ribbons for some rGL
surfaces.  They form a fundamental unit for the translational symmetry group.

\begin{figure}
	\includegraphics[width=0.6\textwidth]{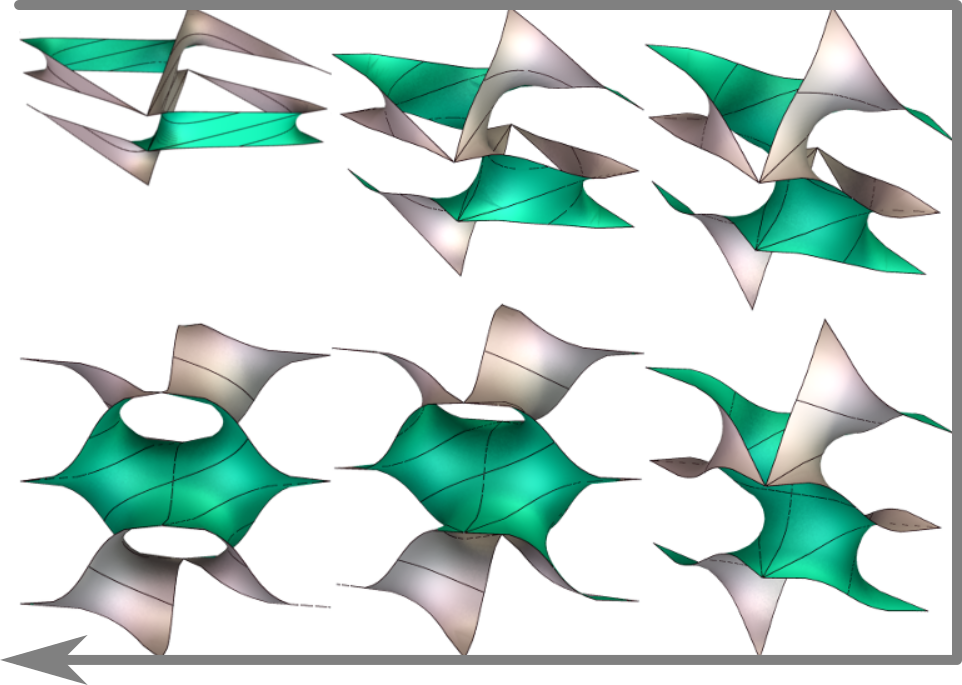}
	\caption{
		rGL surfaces with $\re\tau=-0.95$, $-0.7$, $-0.4$, $-0.1$, $0.2$ and $0.45$,
		in the indicated order.
	}
	\label{fig:rGLgallery}
\end{figure}

\section{Discussion}\label{sec:discuss}

\subsection{Bifurcation}

A TPMSg3 is a \emph{bifurcation instance} if the same deformation of its
lattice could lead to different deformations (bifurcation branches) of the
surface.  Both the tG--tD and the rGL--rPD intersections are bifurcation
instances.

Bifurcation instances among classical TPMSg3s are systematically investigated
in~\cite{koiso2014}, but some of them had no explicit bifurcation branch at the
time.  Two bifurcation instances were discovered in tD~\cite{koiso2014}.  The
recently discovered t$\Delta$ family provides the missing bifurcation branch to
one of them~\cite{chen2018a}.  The other bifurcation instance seems to escape
the attention.  Its conjugate is identified as the tP surface obtained from the
square catenoid of maximum height, but no bifurcation branch was previously
known for itself.  Numerics from~\cite{koiso2014} and~\cite{fogden1999} (see
also~\cite{schroderturk2006}), which we can confirm with the help
of~\eqref{eq:explicitT} in Appendix, shows that this is exactly the tG--tD
intersection with $\tau \approx 1 + 1.51019 i$.  Hence tG provides the missing
bifurcation branch.

Likewise, two bifurcation instances were discovered in rPD~\cite{koiso2014}.
One of them is identified as the rPD surface obtained from the triangular
catenoid of maximum height.  The other is its conjugate, for which no
bifurcation branch was previously known.  Numerics from~\cite{koiso2014}
and~\cite{fogden1999} (see also~\cite{schroderturk2006}) shows that this is
exactly the rGL--rPD intersection, hence rGL provides the missing bifurcation
branch.

Therefore, all bifurcation instances discovered in~\cite{koiso2014} have now an
explicit bifurcation branch.

\begin{remark}
	Curiously, both tG--tD and rGL--rPD intersections are conjugate to catenoids
	of maximum height.  Numerics shows that the height of the twisted catenoid
	seems to increase monotonically with $\re\tau$ along the tG and rGL families,
	and reaches the maximum at the intersection.
\end{remark}

\subsection{Reflection group}

In this part, we point out that ``reflections in classical TPMSg3 families''
generate a reflection group that acts on $\cT$ and $\cR$.

We have seen that each $\tau$ corresponds to a marked twisted catenoid.  The
marked square catenoid is invariant under $\tau \mapsto \tau+4$, and the marked
triangular catenoid is invariant under $\tau \mapsto \tau+3$.

For a twisted square catenoid, as one increases $\re\tau$ beyond $1$, the twist
angle increases beyond $\pi$.  However, note that $\tau \mapsto 2-\tau$ only
results in a reflection of the marked catenoid.  Consequently, their associate
surfaces with the same associate angle differ only by handedness.  So, if
$\tau$ closes the period with associate angle $\theta$ and gives an embedded
TPMSg3 in $\cT$, then $2-\tau$ gives the same TPMSg3 with the same associate
angle.  Similarly, one can decrease $\re\tau$ below $-1$, but $\tau \mapsto
-2-\tau$ only results in a reflection, hence $\tau$ and $-2-\tau$ gives the
same TPMSg3 in $\cT$.

So we have shown that $\cT$ is invariant under the reflections in the vertical
lines $\re\tau = \pm 1$ (tP and tD).  The same argument applies to show that
$\cR$ is invariant under the reflections in $\re\tau = -1$ (H) and $\re\tau =
1/2$ (rPD).

\medskip

Now let us apply the transform $\tau' = 1/\bar\tau$, which exchanges the vertical
lines $\re \tau = \pm 1$ with the half-circles $|\tau \pm 1/2| = 1/2$.  Then
Figure~\ref{fig:tGtau} becomes Figure~\ref{fig:InvtGtau}.

\begin{figure}
	\includegraphics[width=0.5\textwidth]{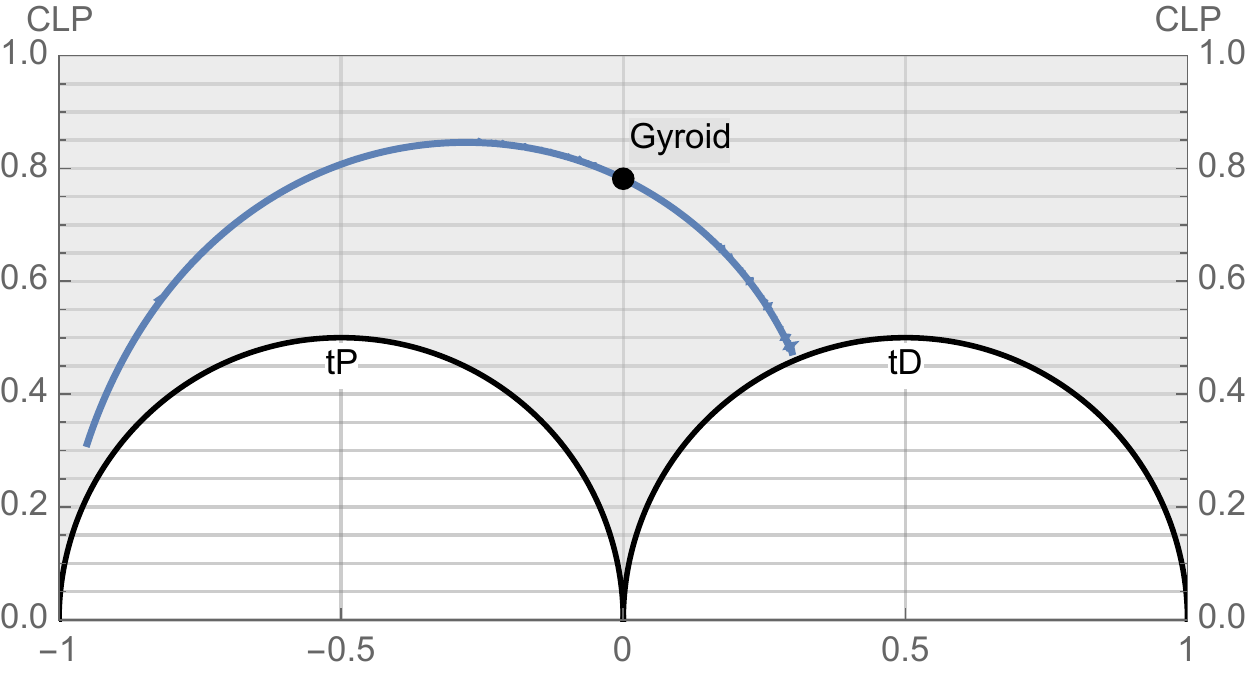}
	\caption{
		Transform of Figure~\ref{fig:tGtau} under the inversion $\tau'=1/\bar\tau$.
	}
	\label{fig:InvtGtau}
\end{figure}

The vertical lines $\re \tau'=\pm 1$ correspond to CLP surfaces.  With such a
$\tau'$ and associate angle $\pi/2$, we know very well that the strip $0 < \im
z < \im \tau'/2$ is mapped by the Weierstrass parameterization to a minimal
strip bounded by two periodic zig-zag polygonal curves related by a vertical
translation.  Each polygonal curve consists of segments of equal length, making
$\pi/2$ angles in alternating directions.

As we increase $\re\tau'$, the straight segments of the polygonal curves become
symmetric curves.  At the same time, the two boundaries begin to drift with
respect to each other.  When $\re\tau' = 0$, the drift is exactly half of a
period, then reflectional symmetry is restored and the segments become straight
again; this is a tD surface.  When $\re\tau' = 1$, the drift is exactly one
period, and a CLP surface is restored.

In principle, we could have built the whole paper on ``drifted strips''
instead of ``twisted catenoids''.  We will not go too far in this direction.
Nevertheless, the alternative view facilitates the following observation: the
transforms $\tau' \mapsto 2-\tau'$ and $\tau' \mapsto -2-\tau'$ only result in
a reflection of the drifted strip.  Consequently, $\cT$ is invariant under the
reflections in $\re\tau' = \pm 1$ (CLP).  Similarly, $\cR$ is invariant under
the reflections in hCLP families ($\re\tau' = \pm 1$).

\medskip

If we parameterize $\cT$ and $\cR$ by $\tau$ in the upper half-plane, which can
be seen as the hyperbolic space, then we have proved that
\begin{proposition}\leavevmode
	\begin{itemize}
		\item $\cT$ is invariant under the discrete group generated by reflections
			in $\re \tau = -1$ (tP), $\re \tau = 1$ (tD) and $|\tau \pm 1/2| = 1/2$
			(CLP).

		\item $\cR$ is invariant under the discrete group generated by reflections
			in $\re \tau = -1$ (H), $\re \tau = 1/2$ (rPD) and $|\tau \pm 1/2| = 1/2$
			(hCLP).
	\end{itemize}
\end{proposition}

\subsection{Higher pitch}

We have seen that a tG surface consists of ribbons bounded by helices, and the
pitch of the helices doubles the vertical distance between them.  In general,
for any surface in $\cT$ and $\cR$, the ratio of the pitch of the helices to
the vertical distance between them must be an even integer, say $2p$; we call
$p$ the \emph{pitch} of the surface.  The vertical associate angle for a
surface of pitch $p$ is $\theta_v(\tau;p) = \arg(\tau + 1 - 2/p)-\pi/2$ for
$\cT$, and $\theta_v(\tau;p) = \arg(\tau + 1 - 3/2p)-\pi/2$ for $\cR$.  The
period condition is again $\theta_h(\tau) = \theta_v(\tau;p)$.

It is immediate that $p=0$ for tP.  We have seen that tD and tG close the
period with $p=1$.  It is also not difficult to see that CLP closes the period
with $p=2$.  In this case, a catenoid and the one directly above it opens into
two ribbons that occupy the same vertical cylinder.  They interlace each other,
and the vertical gap between them equals the vertical width of each single
ribbon, allowing adjacent ribbons to fit exactly in.

More generally, for a surface of pitch $p$, $p$ ribbons occupy the same
vertical cylinder so that adjacent ribbons fit exactly in.  When $p \equiv 0
\pmod 2$, one easily verifies that the half-circle $|\tau + 1 - 1/p| = 1/p$
closes the period for $\cT$ surfaces with pitch $p$.  By the reflection group,
we see that this half-circle corresponds to tP if $p \equiv 0 \pmod 4$ and to
CLP if $p \equiv 2 \pmod 4$.  We now prove that

\begin{lemma}
	Let
	\[
		\tau' = -\frac{(2k-1)\bar\tau + (2k-2)}{2k\bar\tau + (2k-1)}
	\]
	be the image of $\tau$ under the reflection in the half-circle $|\tau+1-1/2k|
	= 1/2k$.  If $\tau$ closes the period (for $\cT$ or $\cR$) with pitch $p$,
	then $\tau'$ closes the period with pitch $q=4k-p$.
\end{lemma}

\begin{proof}
	Note that $2/q-1$ and $2/p-1$ are related by the reflection in the named
	half-circle.  It then follows from elementary geometry that $\theta_v(\tau;p)
	+ \theta_v(\tau';q) = \arg(\tau + 1 - 1/2k)$.  On the other hand, note that
	\[
		\tau'+1 = \frac{\bar\tau + 1}{2k\bar\tau + (2k-1)}.
	\]
	Then a change of variable in the integration $\psi(\tau)$ shows that
	$\theta_h(\tau) + \theta_h(\tau') = -\arg(2k\bar\tau + (2k-1)) = \arg(\tau +
	1 - 1/2k)$.  So $\theta_h(\tau) = \theta_v(\tau;p)$ implies that
	$\theta_h(\tau') = \theta_v(\tau';q)$.
\end{proof}

In particular, by reflections in CLP (resp.\ hCLP), we see that tG and tD
(resp.\ rGL and rPD) close the period for each $p \equiv 1 \pmod 2$.

\subsection{Uniqueness}

We conjecture the following uniqueness statements.
\begin{conjecture}\leavevmode
	\begin{itemize}
		\item	For every $r \in (-1,1)$, there is a unique $\tau$ with $\re\tau = r$
			that solves the period condition~\eqref{eq:period} for $\cT$ with pitch
			$1$.

		\item	For every $r \in (-1,1/2)$, there is a unique $\tau$ with $\re\tau =
			r$ that solves the period condition~\eqref{eq:period} for $\cR$ with
			pitch $1$.
	\end{itemize}
\end{conjecture}
The uniqueness has been proved by Weyhaupt~\cite{weyhaupt2006} for $\cT$ with
$\Re\tau=0$ (gyroid) and for $\cR$ with $\Re\tau=-1/2$ (gyroid) and $0$
(Lidinoid).  This shows that
\begin{theorem}
	The gyroid and the Lidinoid are the only non-trivial embedded TPMSg3s in the
	associate families of tP, tD, rPD and H surfaces.
\end{theorem}
\begin{proof}
	Any other embedded TPMSg3 must, by construction, lie on the same vertical
	lines, hence contradict the uniqueness.
\end{proof}

Our approach leads to a simple proof that works not only for the gyroid and the
Lidinoid but also for the tG surfaces with $\re\tau=\pm 1/2$.  Here is a
sketch: On the one hand, it is immediate that $\theta_v$ decreases
monotonically with $\im\tau$.  On the other hand, by enlarging the outer square
or triangle, and shrinking the inner square or triangle, one easily sees from
the flat structure that $\theta_h$ increases monotonically with $\im\tau$.
This simple proof is possible because, for these cases, the twist angles take
special values and do not vary with $\im\tau$.  This may not hold for other
cases, for which even the meaning of ``twist angle'' is not clear.

\medskip

We also conjecture the following classification statement.

\begin{conjecture}\leavevmode
	\begin{itemize}
		\item	The tP, tD, CLP and tG surfaces are the only members of $\cT$.
		\item The H, rPD and rGL surfaces are the only members of $\cR$.
	\end{itemize}
\end{conjecture}

For a proof, we need to prove the previous conjecture first, then also exclude
the existence of new surface with pitch $0$ or $2$.

Moreover, note that the order-2 rotations around horizontal axes are only used
to determine the L\'opez--Ros factor.  We wonder if this is necessary and
conjecture that

\begin{conjecture}\leavevmode
	\begin{itemize}
		\item	The tP, tD, CLP and tG surfaces are the only TPMSg3s with an order-4 screw symmetry.
		\item The H, rPD and rGL surfaces are the only TPMSg3s with an order-3 screw symmetry.
	\end{itemize}
\end{conjecture}

\appendix

\section{Asymptotic behavior of associate angle}

We now give a detailed asymptotic analysis of $\theta_h$ for $\cT$.

We need to study the shape of the twisted square annulus with more care.  For
safety and convenience, we adopt the natural convention (see
Section~\ref{sec:convention}) for all computations involving Jacobi elliptic
functions.  So we define $\tilde\tau = \tau - r$ where $r := \lfloor \re \tau +
1/2 \rfloor$, hence $-1/2 < \re\tilde\tau \le 1/2$.  We write $\tilde m =
\lambda(2 \tilde \tau)$, and correspondingly $\tilde K = K(\tilde m)$ and
$\tilde K' = - 2 i \tilde \tau \tilde K$.  Note that $\tilde K' = K(1-\tilde
m)$ coincides with the usual definition of the associated complete elliptic
integral of the first kind.

\begin{remark}
The arguments of the Jacobi elliptic functions can be directly replaced by
their tilde versions.  This practice is however not safe elsewhere.  In
particular, the L\'opez--Ros factor $m^{1/4}$ can not be directly replaced by
$\tilde m^{1/4}$.  Instead, we must use the convention that 
\[
	\arg m = 2 \pi r + \arg \tilde m.
\]
\end{remark}

\medskip

Let us first look at
\[
	\psi_1(\tau) := \int_0^{1/2} (m^{1/4} \sn(4\tilde Kz;\tilde \tau))^{1/2} dz,
\]
which is a vector pointing from $\Phi_1(0)$ to $\Phi_1(1/2)$, hence an
straightened edge vector of the inner twisted square.  With the change of
variable $\zeta = \sn(4\tilde Kz;\tilde\tau)$, we have
\begin{equation}\label{eq:psi1}
	\psi_1(\tau) = 2 \frac{m^{1/8}}{4\tilde K} \int_0^1
	\frac{\zeta^{1/2}d\zeta}{\sqrt{(1-\zeta^2)(1-\tilde m\zeta^2)}}.
\end{equation}
Here we used the identities
\begin{align*}
	\frac{d}{dz}\sn(z;\tilde\tau) &= \operatorname{cn}(z;\tilde\tau) \operatorname{dn}(z;\tilde\tau) ;\\
	\operatorname{cn}^2(z;\tilde\tau) &= 1 - \sn^2(z;\tilde\tau) ;\\
	\operatorname{dn}^2(z;\tilde\tau) &= 1 - \tilde m \sn^2(z;\tilde\tau).
\end{align*}

Then we compute the integral
\begin{align*}
	\psi_2(\tau) :=& \int_0^{\tilde\tau/2} (m^{1/4} \sn(4\tilde Kz;\tilde \tau))^{1/2} dz,\\
	=& \int_0^{1/2} (m^{1/4} \sn(4\tilde K\tilde\tau x;\tilde\tau))^{1/2} \tilde\tau dx\\
	=& \int_0^{1/2} (m^{1/4} \sn(2 i \tilde K' x;\tilde\tau))^{1/2} \tilde\tau dx.
\end{align*}
This is a vector pointing from $\Phi_1(0)$ to $\Phi_1(\tilde\tau/2)$, hence
from an inner vertex to the nearest outer vertex of the twisted annulus.
Then we use the Jacobi imaginary transformation $\sn(z;\tau) = -i
\operatorname{sc}(iz;-1/\tau)$ and obtain
\begin{align}
	\psi_2(\tau)= & \int_0^{1/2} (-i m^{1/4} \operatorname{sc}(-2 \tilde K' x;-1/\tilde\tau))^{1/2} \tilde\tau dx\nonumber\\
	= & e^{3\pi i/4} \frac{m^{1/8}}{4\tilde K} \int_0^\infty \frac{\xi^{1/2}d\xi}{\sqrt{(1+\xi^2)(1+\tilde m\xi^2)}},\label{eq:psi2}
\end{align}
where we changed the variable $\xi = -\operatorname{sc}(-2\tilde K'x;-1/\tilde\tau)$ and
used the identities
\begin{align*}
	\frac{d}{dz} \operatorname{sc}(z;\tilde\tau) &= \operatorname{dc}(z;\tilde\tau) \operatorname{nc}(z;\tilde\tau);\\
	\operatorname{dc}^2(z;\tilde\tau) &= 1 + (1-\tilde m) \operatorname{sc}^2(z;\tilde\tau);\\
	\operatorname{nc}^2(z;\tilde\tau) &= 1 + \operatorname{sc}^2(z;\tilde\tau).
\end{align*}

The two vectors $\psi_1(\tau)$ and $\psi_2(\tau)$ can determine the images of
all branch points under $\Phi_1$.

\begin{lemma}\label{lem:tauinf}
	$\theta_h(\tau) \to \re(1-\tau)\pi/4$ as $\im \tau \to \infty$.
\end{lemma}

\begin{proof}
	As $\im \tau = \im \tilde \tau \to \infty$, we have~\cite[(2.1.12)]{lawden1989}
	\[
		m(\tau) = \lambda(2\tau) \sim 16 \exp(2i\pi\tau),
	\]
	so $|m| \to 0$.  Recall that $\tilde K(0) = \pi/2$, and that the integral
	in~\eqref{eq:psi1} tends to
	\[
		\int_0^1 \frac{\zeta^{1/2}d\zeta}{\sqrt{(1-\zeta^2)}}
	\]
	which is bounded.   Therefore, $\psi_1(\tau) \to 0$ as $\im \tau \to \infty$.  In
	other words, the size of the inner square tends to $0$.

	On the other hand, we have
	\[
		\psi_2(\tau) \sim e^{3i\pi/4} \frac{m^{1/8}}{4\tilde K} \int_0^\infty \frac{d\xi}{\xi^{1/2}\sqrt{1+\tilde m\xi^2}}
		= e^{(r+3)i\pi/4} \frac{\tilde m^{-1/8}}{4\tilde K} \int_0^\infty \frac{du}{u^{1/2}\sqrt{1+u^2}},
	\]
	where we changed the variable $u = \xi \sqrt{m}$ and used the convention that
	$\arg m = 2\pi r + \arg \tilde m$.  Note again that the integral is bounded,
	hence $|\psi_2(\tau)| \sim |m|^{-1/8}$.  In other words, the size of the outer
	square grows exponentially with $\im \tilde \tau$.

	Therefore, as $\im\tau \to \infty$, the integral
	\[
		\psi(\tau) = \int_0^{(1+\tau)/2} (m^{1/4} \sn(4Kz; \tau))^{1/2} dz
	\]
	is dominated by
	\begin{align*}
		\psi(\tau) \sim&\int_{(r+1)/2}^{(\tau+1)/2} (m^{1/4} \sn(4Kz; \tau))^{1/2} dz \sim e^{-(r+1)i\pi/2} \psi_2(\tau)\\
		\sim &e^{-(r+1)i\pi/2} e^{(r+3)i\pi/4} \frac{\tilde m^{-1/8}}{4\tilde K} \int_0^\infty \frac{du}{u^{1/2}\sqrt{1+u^2}},
	\end{align*}
	Now we
	can conclude that
	\[
		\theta_h(\tau) = \arg \psi(\tau) \to (1-r-\re\tilde\tau)\pi/4 = \re(1-\tau)\pi/4.
	\]
\end{proof}

A similar argument applies to rGL, so we omit the proof.  The conclusion is
that
\[
	\theta_h \to \re(1/2-\tau)\pi/3 \quad\text{as $\im \tau \to \infty$.}
\]

\begin{lemma}\label{lem:tau0}
	$\theta_h(\tau) \to 0$ from the negative side as $\tau \to 1$ within $\Omega_t$.
\end{lemma}

\begin{proof}
	By the transformation $\lambda(-1/\tau) = 1 -
	\lambda(\tau)$~\cite[(9.4.10)]{lawden1989}, we see that
	\[
		\tilde m = m(\tilde \tau) = 1 - 16 e^{-i\pi/2\tilde\tau} +
		O(e^{-i\pi/\tilde\tau}) \qquad\text{as }\tilde\tau \to 0.
	\]
	Then~\cite[Exercise 8.13]{lawden1989}
	\[
		\tilde K = \int_0^1 \frac{d\zeta}{\sqrt{(1-\zeta^2)(1-\tilde m\zeta^2)}} \sim
		-\frac{1}{2} \ln(1-\tilde m) \sim \frac{i\pi}{4\tilde \tau}.
	\]
	The standard proof for this also applies to prove that the integral
	in~\eqref{eq:psi1} satisfies
	\[
		\int_0^1 \frac{\zeta^{1/2}d\zeta}{\sqrt{(1-\zeta^2)(1-\tilde m\zeta^2)}} \sim
		-\frac{1}{2} \ln(1-\tilde m) \sim \frac{i\pi}{4\tilde \tau}.
	\]
	One can also quickly convince oneself by noting that the integrand
	in~\eqref{eq:psi1} differs from the integrand of $\tilde K$ only by a factor
	$\zeta^{1/2}$, which can be neglected near $1$, where the divergence occurs.
	In other words, the integral in~\eqref{eq:psi1} is asymptotically equivalent
	to $\tilde K$.  So we have $\psi_1(\tau) \to \frac{1}{2}m^{1/8}$ as $\tau \to 1$.

	On the other hand, the integral in~\eqref{eq:psi2} tends to
	\[
		\int_0^\infty \frac{\xi^{1/2}d\xi}{1+\xi^2} = \frac{\pi}{\sqrt{2}}
	\]
	which is bounded, hence $\psi_2(\tau) \to 0$.  Therefore, as $\tau \to 1$, we
	see from the flat structure that $\psi(\tau) \to \frac{1-i}{2} m^{1/8}$, so
	\[
		\theta_h(\tau) = \arg(m(\tau))/8-\pi/4 = \arg(\tilde m)/8 \to 0.
	\]
	Note again that we used the convention $\arg m(\tau) = m(\tilde\tau + 1) =
	2\pi + \arg m(\tilde\tau)$.

	We now prove that the convergence is from the negative side.  A routine
	calculation shows that
	\[
		\arg(e^{-i\pi/2\tilde\tau}) = -\frac{\pi\re\tilde\tau}{2|\tilde\tau|^2} = \frac{\pi\sin\theta_v}{2|\tilde\tau|}.
	\]
	When $\tau \in \Omega_t$, it follows from elementary geometry that $0 <
	\sin\theta_v < |\tilde\tau|$, hence $0 < \arg(e^{-i\pi/2\tilde\tau}) <
	\pi/2$.  Therefore, when $\tau$ tends to $0$ within $\Omega_t$, we have
	$\arg(\tilde m) = \arg(1-16 e^{-i\pi/2\tilde\tau}) < 0$.
\end{proof}

\begin{remark}
	The integrals arisen from $\cT$ surfaces (e.g.~\eqref{eq:psi1}
	and~\eqref{eq:psi2}) can be explicitly expressed in terms of elliptic
	integrals; see~\cite[Chapter X]{bowman1961} and~\cite[\S 595 et
	seq.]{byrd1971}.  For example, we have
	\begin{align*}
		\int_0^{1/2} (m^{1/4} \sn(4\tilde Kz;\tilde \tau))^{1/2} dz =& \frac{e^{r \pi i/4}}{2\sqrt{2}}\frac{\tilde m^{-1/8}}{\sqrt{1+\tilde m^{1/2}}}\frac{K(\mu)-K'(\mu)}{K(\tilde m)},\\
		\int_{\tau/2}^{1/2+\tau/2} (m^{1/4} \sn(4\tilde Kz;\tilde \tau))^{1/2} dz =& \frac{e^{r \pi i/4}}{2\sqrt{2}}\frac{\tilde m^{-1/8}}{\sqrt{1+\tilde m^{1/2}}}\frac{K(\mu)+K'(\mu)}{K(\tilde m)},
	\end{align*}
	for the inner and outer edge vectors of the twisted annulus, where
	\[
		\mu = \frac{(1+\tilde m^{1/4})^2}{2+2\tilde m^{1/2}}.
	\]
	We then obtain that
	\begin{equation}\label{eq:explicitT}
		\psi(\tau) = e^{r \pi i/4} \frac{1-i}{2\sqrt{2}} \frac{\tilde m^{-1/8}}{\sqrt{1+\tilde m^{1/2}}}\frac{K(\mu)}{K(\tilde m)}.
	\end{equation}
	This facilitates the numeric calculation for the tG--tD intersection but does
	not give an explicit expression.

	We are not aware of equally explicit expressions for integrals arisen from
	$\cR$ surfaces.
\end{remark}

\bibliography{References}
\bibliographystyle{alpha}

\end{document}